\documentclass[10pt,a4paper]{amsart}

\usepackage{verbatim}
\usepackage{geometry} 
 \newgeometry{tmargin=2.5cm, bmargin=3cm, lmargin=2.5cm, rmargin=2.5cm}
\usepackage[T1]{fontenc}
\usepackage[notref,notcite]{}
\usepackage{graphicx}
\usepackage{enumerate}
\usepackage{amsmath,amsfonts,amssymb}
\usepackage{mathtools}
\usepackage{ulem}
\usepackage{color}
\usepackage{cite}
\definecolor{citation}{rgb}{0.2,0.58,0.2} 
\definecolor{formula}{rgb}{0.1,0.2,0.6}
\definecolor{url}{rgb}{0.3,0,0.5}
\usepackage{pgf,tikz}
\usepackage{mathrsfs}
\usepackage{fancyhdr}
\usepackage{dsfont}
\usepackage{esint}
\usepackage{hyperref}
\usepackage[centertags]{amsmath}

\newtheorem{coro}{\bf Corollary}[section]
\newtheorem{theo}[coro]{\bf Theorem} 
\newtheorem{lem}[coro]{\bf Lemma}
\newtheorem{rem}[coro]{\bf Remark} 
\newtheorem{defi}[coro]{\bf Definition}

\newtheorem{prop}[coro]{\bf Proposition}
\newtheorem{theorem}{\bf Assumption} 
\makeatletter
\newcommand{\settheoremtag}[1]{
  \let\oldthetheorem\thetheorem
  \renewcommand{\thetheorem}{#1}
  \g@addto@macro\endtheorem{
    \addtocounter{theorem}{-1}
    \global\let\thetheorem\oldthetheorem}
  }
\makeatother

\definecolor{darkgreen}{rgb}{0.00, 0.50, 0.00}

\DeclareMathOperator*{\osc}{osc}

\newcommand{\nocontentsline}[3]{}
\newcommand{\tocless}[2]{\bgroup\let\addcontentsline=\nocontentsline#1{#2}\egroup}

\def\tens#1{\pmb{\mathsf{#1}}}

\def\mean#1{\mathchoice%
          {\mathop{\kern 0.2em\vrule width 0.6em height 0.69678ex depth -0.58065ex
                  \kern -0.8em \intop}\nolimits_{\kern -0.4em#1}}%
          {\mathop{\kern 0.1em\vrule width 0.5em height 0.69678ex depth -0.60387ex
                  \kern -0.6em \intop}\nolimits_{#1}}%
          {\mathop{\kern 0.1em\vrule width 0.5em height 0.69678ex
              depth -0.60387ex
                  \kern -0.6em \intop}\nolimits_{#1}}%
          {\mathop{\kern 0.1em\vrule width 0.5em height 0.69678ex depth -0.60387ex
                  \kern -0.6em \intop}\nolimits_{#1}}}
          
\newcommand{\dv}{{\rm div}}

\newcommand{\opA}{{\mathcal{ A}}}

\newcommand{\bopA}{{\bar{\opA}}}
\newcommand{\wt}{\widetilde}
\newcommand{\ve}{\varepsilon}
\newcommand{\vp}{\varphi}
\newcommand{\vt}{\vartheta}
\newcommand{\vr}{\varrho}
\newcommand{\vk}{\varkappa}

\def\R{{\mathbb{R}}}

\def\rp{{[0,\infty)}}

\def\tew{{\tens{w}}}
\def\teeta{{\tens{\eta}}}
\def\texi{{\tens{\xi}}}
\def\teze{{\tens{\zeta}}}
\def\teu{{\tens{u}}}
\def\teW{{\tens{W}}}
\def\teh{{\tens{h}}}
\def\tebu{{\overline{\tens{u}}}}

\def\tev{{\tens{v}}}

\def\calV{{\mathcal{V}}}
\def\tebv{{\overline{\tens{{v}}}}}
\def\tevp{{\tens{\vp}}}

\def\tef{{\tens{f}}}

\def\teell{{\tens{\ell}}}

\def\sa{{s_{\rm app}}}
\def\cL{{\mathcal{L}}}
\def\dx{{\,\mathrm{d}x}}
\def\dy{{\,\mathrm{d}y}}
\def\ds{{\,\mathrm{d}s}}
\def\dtau{{\,\mathrm{d}\tau}}
\def\dt{{\,\mathrm{d}t}}

\def\teDu{{D\teu}}
\def\teDbu{{D\tebu}}

\def\teDv{{D\tev}}
\def\teDw{{D\tew}}
\def\teDbv{{D\tebv}}

\def\teDvp{D{\tevp}}

\def\cI{{\mathcal{I}}}

\def\cB{{\mathcal{B}}}

\def\EDbv{{E_{\teDbv}}}

\def\tedv{{\tens{\dv}}}
\def\temu{{\tens{\mu}}}
\def\btemu{{\overline{\temu}}}

\def\rn{{\mathbb{R}^{n}}}
\def\Rm{{\mathbb{R}^{m}}}
\def\Rn{{\mathbb{R}^{n}}}
\def\id{{\mathsf{Id}}}
\def\P{{\mathsf{P}}}

\def\rnm{{\mathbb{R}^{n\times m}}}
 
\def\Mb{{\mathcal{M}(\Omega,\Rm)}} 

\newcommand{\data}{\textit{\texttt{data}}}

\title[Gradient Riesz potential estimates for Orlicz systems]{Gradient Riesz potential estimates\\ for a general class of measure data quasilinear systems}

\author{Iwona Chlebicka}\address{Iwona Chlebicka \\
Institute of Applied Mathematics and Mechanics, University of Warsaw \\ ul. Banacha 2, 02-097 Warsaw, Poland\\  \texttt{e-mail: i.chlebicka@mimuw.edu.pl}} 

\author{Minhyun Kim}\address{Minhyun Kim\\Department of Mathematics \& Research Institute for Natural Sciences, Hanyang University, 04763 Seoul, Republic of Korea\\ \texttt{e-mail: minhyun@hanyang.ac.kr}}

\author{Marvin Weidner}\address{Marvin Weidner \\ Departament de Matemàtiques i Informàtica, Universitat de Barcelona, Gran Via de les Corts Catalanes 585, 08007 Barcelona, Spain\\ \texttt{e-mail: mweidner@ub.edu}}

\newcommand\blfootnote[1]{%
  \begingroup
  \renewcommand\thefootnote{}\footnote{#1}%
  \addtocounter{footnote}{-1}%
  \endgroup
}

\begin{document}

\makeatletter
\@namedef{subjclassname@2020}{\textup{2020} Mathematics Subject Classification}
\makeatother
\blfootnote{Data sharing not applicable to this article as no datasets were generated or analysed during the current study.}

\maketitle \sloppy

\thispagestyle{empty}

\belowdisplayskip=18pt plus 6pt minus 12pt \abovedisplayskip=18pt
plus 6pt minus 12pt
\parskip 4pt plus 1pt
\parindent 0pt

\parindent 1em

\begin{abstract}
We study the gradient regularity of solutions to measure data elliptic systems with Uhlenbeck-type structure and Orlicz growth. For any bounded Borel measure, pointwise estimates for the gradient of solutions are provided in terms of the truncated Riesz potential. This allows us to show a precise transfer of regularity from data to solutions on various scales.
\end{abstract}

\setcounter{tocdepth}{1}

\tableofcontents

\section{Introduction}

\subsection{Objectives}

The goal of this article is to establish fine local regularity properties of solutions ${\teu}:\Omega \subset \Rn\to\Rm$ ($n \geq 2$, $m \geq 1$) to quasilinear elliptic systems with measure data of the following form
 \begin{equation}
\label{eq:mu}
-{ \tedv}\left( \frac{g(|\teDu|)}{|\teDu|}  \teDu \right)=\temu\quad\text{in}~ \Omega\,,
\end{equation}
where $\Omega$ is open in $\Rn$, $\tedv$ stands for the $\Rm$-valued divergence operator, $\temu\in\mathcal{M}(\Omega,\Rm)$ is a bounded Borel measure, and $g$ is a function satisfying certain Orlicz-type growth conditions. In particular, we are interested in the precise transfer of regularity properties from the data $\temu$ to the solution $\teu$. In order to achieve this, we establish a pointwise estimate for $\teDu$ in terms of the truncated Riesz potential of $\temu$ (see~\eqref{Riesz-potential}) reading
\begin{equation}
\label{intro:eq:u-est}  g(|\teDu(x_0 )|)\lesssim \cI^{|\temu|}_1 (x_0 ,r)+ g\left(\mean{B_r(x_0)} |\teDu(x)|\dx\right)\,.
\end{equation}

Our framework is fairly general and includes for instance the following special cases: $p$-Laplacian, when $g(t)=t^{p-1}$ and $p>2$, as well as operators generated by Zygmund-type functions $g_{p,\alpha}(t)=t^{p-1}\log^\alpha(c+t)$ for $p>2$, $\alpha\in\R$, and sufficiently large $c$.

Potential estimates of the form \eqref{intro:eq:u-est} are an adequate nonlinear replacement of integral representation formulas in terms of fundamental solutions. Such formulas hold true for linear equations like the scalar Poisson equation $-\Delta u = \mu$ but cannot be expected in the realm of nonlinear problems. A main achievement of the research during  the past thirty years has been to establish nonlinear potential estimates in terms of Wolff and Riesz potentials for solutions and gradients of solutions to nonlinear problems of the form
\begin{equation}
\label{eq:mu-scalar} 
-\dv\, \opA(Du)=\mu \quad\text{in}~ \Omega\,
\end{equation}
involving $\opA$ of various rate of degeneracy. Potential estimates can be applied to infer fine information about solutions and allow to investigate sharp regularity properties of solutions to \eqref{eq:mu-scalar} in terms of the inhomogeneity $\mu$. Another challenge  comes from the presence of general measure data $\mu$, since in this setting the classical notion of weak solutions turns out not to be suitable for the analysis of \eqref{eq:mu-scalar}. Instead, one employs the weaker solution concept of SOLA (solutions obtained as limits of approximations) introduced in~\cite{BG}. These solutions are less controlled making their analysis more delicate. 

The landmark paper \cite{KiMa92}, where potential estimates have been established for equations driven by the $p$-Laplacian,  has been followed by a stream of deep results, see e.g.~\cite{KiMa94,KuMi2013,DuMi2010}. Equations of non power-type growth like~\eqref{eq:mu} are studied since~\cite{Go,Ta,lieb}. Potential estimates for solutions to scalar problems with Orlicz growth were provided in~\cite{CGZG-Wolff,Maly-Orlicz}. On the other hand, one can analyze the regularity transfer of the inhomogeneity not only to the solution, but also to its gradient. In the super-quadratic case, gradient bounds for solutions to measure data equations were first provided in terms of nonlinear Wolff-type potentials \cite{DuMi2011,DuMi2010,DuMi2010-2}. Later, improved estimates in terms of Riesz potentials were found to hold true (see \cite{KuMi2013}). Gradient potential estimates for solutions to \eqref{eq:mu-scalar} with Orlicz growth functions $g$ were established in~\cite{Baroni-Riesz} and extended in \cite{XZM}. We refer to~\cite{KuMi2014,KuMi2016} for a rather complete account on the state of art.

While the nonlinear potential theory for scalar equations of the form \eqref{eq:mu-scalar} is by now fairly well-understood, some questions remained open in the development of an analogous theory for nonlinear systems. The analysis of systems comes along with some significant technical difficulties, which are for instance caused by the lack of a maximum principle. Moreover, controlling the energy of solutions is very hard if a system is not of a particular structure. The regularity theory for homogeneous nonlinear systems of quasi-diagonal structure modeled upon the $p$-Laplacian has been initiated by Uhlenbeck and Uralt'seva in their celebrated works~\cite{Ural,Uhl}.

Despite the aforementioned serious challenges, the potential theory for nonlinear vectorial problems remains valid in some special cases. In particular, let us refer to~\cite{CiSch} for Wolff potential estimates for solutions to $p$-Laplace systems with datum in divergence form, to~\cite{KuMi2018} for estimates for solutions and their gradients to $p$-Laplace problems with measure datum in terms of Wolff and Riesz potentials, respectively, and to~\cite{CYZG} for generalized Wolff potential estimates for solutions to measure data problems exposing Orlicz growth. The goal of this article is to extend the nonlinear potential theory for vectorial $p$-Laplace problems to nonlinear systems subject to more general growth conditions.

\subsection{Main results and their consequences }\label{ssec:formulation-n-reg-cons}
Complementing the study~\cite{Baroni-Riesz,DuMi2011,KuMi2013,KuMi2018,CiSch,CYZG}, in this article we provide the Riesz potential estimate~\eqref{intro:eq:u-est} for gradients of solutions to problems like~\eqref{eq:mu} possessing Orlicz growth. We stress that such estimates are already known in the scalar case~\cite{Baroni-Riesz} and that our result is an exact and precise analogue of the work~\cite{KuMi2018} for vectorial problems with power growth.

Let us define a vector field $\opA:\rnm\to\rnm$ by
\begin{equation}
    \label{opA:def}\opA(\texi):=\frac{g(|\texi|)}{|\texi|}\,\texi\,.
\end{equation}
With this definition \eqref{eq:mu} can be rewritten as a problem of the form $-\tedv\, \opA(\teDu) = \temu$. Throughout the rest of the article we will impose the following assumption on the function~$g$.

\settheoremtag{(G)}
\begin{theorem} \label{ass:G}
We say that $g : [0,\infty) \to [0,\infty)$ satisfies Assumption~\ref{ass:G} if there exist $2<p\leq q<\infty$ and an $N$-function $G \in C^3((0,\infty))\cap C(\rp)$, i.e. $G$ is convex, vanishes only at $0$, and satisfies 
$\lim _{t \to 0}{G (t)}/{t}=0$ {and} $\lim _{t \to \infty
}{G (t)}/{t}=\infty$, such that $g=G'$ and such that for all $t \geq 0$:
\begin{align} 
\label{a:p}
pG(t) &\leq t g(t) \leq qG(t)\,,\\
\label{a:p-1}
(p-1)g(t) &\leq t g'(t) \leq (q-1)g(t)\,,\\
\label{a:p-2}
(p-2) g'(t) &\leq tg''(t) \leq (q-2) g'(t)\,.
\end{align}
\end{theorem}

As is common in the literature on measure data problems, we consider solutions obtained as limits of approximations (SOLA), defined as follows:

\begin{defi}[Solutions obtained as limits of approximations (SOLA)]
Suppose that Assumption~\ref{ass:G} holds true. A map $\teu\in {W^{1,1}(\Omega,\Rm)}$ such that $\int_\Omega g(|\teDu|)\dx<\infty$ is called a SOLA to~\eqref{eq:mu} if there exists a sequence $(\teu_{h})\subset W^{1,G }(\Omega,\Rm)$ of weak solutions to the systems
\begin{equation*}
-{ \tedv}\left( \frac{g(|D\teu_{h}|)}{|D\teu_{h}|}  D\teu_{h} \right) = {\temu_h} \quad\text{in}~\Omega
\end{equation*}
such that $\teu_{h}\to \teu$ locally in $W^{1,1}(\Omega,\Rm)$, where $(\temu_h)\subset L^\infty (\Omega,\Rm)$ is a sequence
of maps that converges to $\temu$ weakly in the sense of measures and satisfies
\begin{equation}
    \label{conv-of-meas}
\limsup_h |\temu_h |(\overline{B}) \leq |\temu|(\overline{B})\quad\text{
for every ball $B\subset\Omega$\,.}
\end{equation}
\end{defi}
\noindent The existence of SOLA to \eqref{eq:mu} is established in~\cite[Theorem~1.2]{CYZG-Ex} for any measure datum $\temu \in \Mb$. 

We provide pointwise gradient estimates in terms of truncated Riesz potentials given by
  \begin{equation}
\label{Riesz-potential}
\cI^{|\temu|}_1(x_0,R):=\int_0^R  \frac{|\temu|(B_r(x_0))}{r^{n-1}}\,\mathrm{d}r\,.
\end{equation}  
As mentioned before, in the superquadratic case, Riesz potentials are in general smaller than Wolff-potentials. Therefore, estimates from above involving Riesz potential are more precise.

We are now ready to state the two main results of this article.

 \begin{theo}[Pointwise Riesz potential estimates]\label{theo:pointwise}
    Let $\teu:\Omega\to\Rm$ be a SOLA to~\eqref{eq:mu} for some $\temu\in\Mb$ and suppose that Assumption~\ref{ass:G} holds true. Let $B_r(x_0)\Subset\Omega$. If $\cI^{|\temu|}_1(x_0,r)$ is finite, then $x_0$ is a Lebesgue's point of $\teDu$ and\begin{equation}
        \label{Riesz-osc-est}
        g(|\teDu(x_0)-(\teDu)_{B_r(x_0)}|)\leq C_\cI \cI^{|\temu|}_1(x_0,r)+C_{\cI} g\left(\mean{B_r(x_0)}|\teDu-(\teDu)_{B_r(x_0)}|\dx\right)
    \end{equation}
holds for some $C_\cI=C_\cI(n,m,p,q,G(1),g(1),g'(1),g''(1))>0$. In particular, we have the following pointwise estimate:
\begin{flalign}
\label{eq:u-est}  g(|\teDu(x_0 )|)\leq C_{\cI} \cI^{|\temu|}_1 (x_0 ,r)+C_{\cI}g\left(\mean{B_r(x_0)} |\teDu(x)|\dx\right)\,.
\end{flalign}
\end{theo}

Note that the first claim in Theorem~\ref{theo:pointwise} is a consequence of the first part of the following theorem. The second part provides a criterion for the gradient of solution to be continuous.

  \begin{theo}[VMO and continuity criteria]\label{theo:VMOe}
    Let $\teu:\Omega\to\Rm$ be a SOLA to~\eqref{eq:mu} for some $\temu\in\Mb$ and suppose that Assumption~\ref{ass:G} holds true. Let $x_0 \in \Omega$. Then, we have:
    \begin{enumerate}[(i)]
        \item
        If  $\ {\displaystyle      \lim_{\vr\to 0}{|\temu|(B_\vr(x_0))}{\vr^{1-n}}=0}\,$, 
    then $\teDu$ has vanishing mean oscillation at $x_0$, i.e. 
    \begin{equation*}
        \lim_{\vr\to 0}\mean{B_\vr(x_0)}|\teDu-(\teDu)_{B_\vr(x_0)}|\dx=0\,.
    \end{equation*}
    \item
    If $\ \displaystyle\lim_{\vr\to 0}\sup_{x\in B_r(x_0)}\cI^{|\temu|}_1(x,\vr)=0\,$, 
then $\teDu$ is continuous in $B_r(x_0)\,.$
    \end{enumerate}
\end{theo} 

 \begin{coro}\label{coro:Stein}
Suppose Assumption~\ref{ass:G} is satisfied  and $\teu : \Omega \to \Rm$ is a SOLA  to~\eqref{eq:mu}. 
 If $|\temu| \in \Lambda^{1, n}(\Omega)$ (Lorentz space defined in Appendix~\ref{ssec:fnsp}), then $\teDu$ is locally bounded and continuous in $\Omega$. 
\end{coro}
This fact is known as {\it the nonlinear Stein theorem}. We can infer it directly from Theorem~\ref{theo:VMOe} {\it (ii)}, as  $|\temu| \in \Lambda^{1, n}(\Omega)$ implies that $\lim_{\vr\to 0}\sup_{x\in B_r(x_0)}\cI^{|\temu|}_1(x,\vr)=0\,$. Related results for variational minimizers are proven in \cite[Theorem~1.15]{bemi}, for their global counterpart we refer to~\cite{CiMa-ARMA2014}.

A central advantage of Theorem~\ref{theo:pointwise} is that it reduces the problem of establishing gradient estimates for solutions to~\eqref{eq:mu} to the study of Riesz potentials. Since the boundedness of Riesz potentials is well understood in various classical function spaces, see e.g.~\cite{Adams-Riesz,Ci-pot}, we can directly infer multiple regularity properties of gradients of solutions to~\eqref{eq:mu} from the potential estimate in Theorem~\ref{theo:pointwise}. We summarize the consequences of the potential estimate in the following corollary. Note that the definitions of the Lorentz spaces $\Lambda^{r,s}$, the Morrey spaces $L_\vt^r$, the Lorentz--Morrey spaces $\Lambda_\vt^{r,s}$, and the Orlicz space $L\log L$ are recalled in Appendix~\ref{ssec:fnsp}.

 \begin{coro}\label{coro:reg}
Suppose Assumption~\ref{ass:G} is satisfied  and $\teu : \Omega \to \Rm$ is a SOLA  to~\eqref{eq:mu}.  Then the following implications hold true.\begin{enumerate}[(i)]
    \item  If $1 < r <n$, $0 <s\leq\infty$, and $|\temu|\in\Lambda^{r,s}_{\mathrm{loc}}(\Omega)$, then $g(|\teDu|)\in \Lambda^{\frac{r n}{n-r},s}_{\mathrm{loc}}(\Omega)$\,.
   \item   If $|\temu|\in L^1_{\mathrm{loc}} (\Omega)$, then $g(|\teDu|)\in \Lambda^{\frac{n}{n-1},\infty}_{\mathrm{loc}}(\Omega)$\,.
   \item  If locally $|\temu|\in L\log L (\Omega)$, then $g(|\teDu|)\in L_{\mathrm{loc}}^{\frac{n}{n-1}}(\Omega)$\,.
   \item If $1 < r < \vt \leq n$ and locally $|\temu| \in L^{r}_{\vt}(\Omega)$, then locally $g(|\teDu|) \in L_{\vt}^{\frac{\vt r}{\vt - r}}(\Omega)$\,.
   \item  If $1 < r < \vt \leq n$, $0<s<\infty$, 
 and locally $|\temu| \in \Lambda^{r, s}_{\vt}(\Omega)$, then $g(|\teDu|) \in \Lambda^{\frac{\vt r}{\vt - r}, \frac{\vt s}{\vt - r}}_{\mathrm{loc}}(\Omega)$\,.
\end{enumerate}
\end{coro}

Similar families of results on regularity consequences of potential estimates can be found in \cite[Section~4]{Ci-pot},  \cite[Corollary~1]{KuMi2014}, and \cite[Section~10]{KuMi2018} for power-growth problems, and in \cite[Section~5.3]{Bor-Chl-Potentials} for scalar Orlicz problems.  

There is a wide range of literature on regularity properties for solutions to nonlinear problems in the spirit of Corollary~\ref{coro:reg} not relying on techniques from nonlinear potential theory. Let us give a brief overview of the results in this direction. We refer to~\cite{IC-gradest} for gradient regularity in the case of Lorentz--Morrey data (by the methods of~\cite{min-grad-est}) and to~\cite{CiMa} for data in the generalized Marcinkiewicz space. Precise scales ensuring (H\"older) continuity of solutions are presented in~\cite{CGZG-Wolff} in the scalar case and in~\cite{CYZG} for systems. We refer to \cite{BCDS,BCDKS,CiSch} for transfer of regularity from divergence form data in $p$-Laplace problems.  Moreover, second-order regularity properties for Orlicz systems can be found in~\cite{BCDM}. Finally, let us mention  \cite{DiMa,MarcPapi,Marc1996,str,bemi} for regularity results similar to Corollary~\ref{coro:reg} in the setting of systems subject to more general growth conditions. To our best knowledge potential estimates for measure data problems in such a general setup are not known.

Let us point out that implications of potential estimates are not restricted to the spaces mentioned in Corollary~\ref{coro:reg}. In fact, a lot more results on the transfer of regularity from locally integrable data to solutions can be proved by using the estimate $
    \int_{\Omega} |f(y)|\dy \leq \int_{0}^{|\Omega|} f^{*}(s)\ds$, together with the rearrangement inequality (see~\cite{oneil}):
\begin{equation*}
(\cI_{\alpha}f)^*(t)\leq C\int_t^\infty s^{\frac{\alpha}{n}-1}f^{**}(s)\ds\,,\quad\text{for some } \ C=C(\alpha,n) > 0\,.\end{equation*}
From here, to show boundedness of $g(|\teDu|)$ in any rearrangement invariant space, it suffices to apply a~certain Hardy-type inequality to the right-hand side in the display above. For results in this direction, we refer e.g. to ~\cite{Ci-strong-weak,CiMa-JEMS}. 

\subsection{Methods} Many techniques that are commonly used for scalar equations cannot be adapted to systems of PDEs. This results from deep differences between  scalar and vectorial problems like the lack of a comparison principle for solutions to systems. Moreover, continuity of a solution to an equation can be proven for measurable and bounded coefficients, whereas for systems one needs their continuity. In turn, despite it was expected since~\cite{KuMi2013} that $p$-Laplace systems should enjoy a potential estimate of the type~\eqref{intro:eq:u-est}, it required a novel approach and was established in~\cite{KuMi2018}. The main idea of the latter contribution is to employ several technical ideas from the classical potential theory concerning measure data problems and to use a special kind of $p$-harmonic approximation that is a common tool in the partial regularity theory, see~\cite{DuGr,DuMi2004}. In this article, we develop a general growth version of the results in \cite{KuMi2018}. Note that this problem is mentioned as an interesting and far from trivial task in ~\cite[Section~11]{KuMi2018}.

The methods we apply in this paper can be seen as a careful refinement of the technique introduced in~\cite{KuMi2018}. Adapting the approach to the Orlicz setting comes with several technical difficulties.  One of the main tools is the $\opA$-harmonic approximation result due to~\cite{CYZG}, whose proof is also based on the ideas of \cite{KuMi2018}. Let us note that there are other $\opA$-harmonic approximation lemmas in the literature for related problems with Orlicz growth, e.g.~\cite{FILV,DSV3}, but they are not suitable for treating measure data problems. On the other hand, it should be stressed that the proofs not only allow for treating arbitrary measure data, but at the same time the results are obtained precisely in the scale relevant for the operator. The method we develop is delicate enough to keep such precision.

The proof goes as follows. In Section~\ref{sec:comparison}, we provide some comparison results for weak solutions, which are key ingredients for the proofs of our main results. In order to establish the pointwise gradient estimate at a Lebesgue's point $x_0$, we divide into two cases --- degenerate and nondegenerate case --- depending on the size of the gradient average at $x_0$ compared to the so-called excess functional. When the size of the gradient average is smaller than the excess functional, the system \eqref{eq:mu} is considered as a degenerate system, and otherwise as a nondegenerate system. In each case, we establish comparison estimates, which provide regularity transfer from approximations to solutions. In the degenerate case,  by using an $\opA$-harmonic approximation result (Theorem~\ref{theo:Ah-approx}), we construct a solution to a nonlinear homogeneous equation (an $\mathcal{A}$-harmonic map) which is close to the solution in a suitable sense. In the nondegenerate case we compare the original solution with a solution to the linearized system. These elaborate estimates, taking the Orlicz growth into account, are of independent interest. Then, in Section~\ref{sec:SOLA}, we obtain the excess decay for SOLA by combining the two cases described above and passing to the limit. This requires the analysis of delicate estimates on concentric balls. Finally, we can conclude the proof of our main results Theorem~\ref{theo:pointwise}  and Theorem~\ref{theo:VMOe} in Section~\ref{sec:main-proofs}.  

\section{Preliminaries}\label{sec:prelim}

\subsection{Notation}
We shall adopt the convention of denoting by $c$ or $C$ a constant that may vary from line to line. In order to shorten notation, we collect the dependencies of certain constants on the parameters of our problem as $\data=\data(n,m,p,q,G(1),g(1),g'(1),g''(1))$. By $a \lesssim b$ we mean $a \leq Cb$ for some $C=C(\data)$. By $a\approx b$ we mean $a\lesssim b$ and $b\lesssim a$.

For a measurable set $U\subset \rn$ with $0<|U|<\infty$ and a measurable map $\tef\colon U\to \R^{k}$, $k\ge 1$, we define
\begin{equation*}
(\tef)_U:=\mean{U}\tef(x) \dx :=\frac{1}{|U|}\int_{U}\tef(x) \dx\quad\text{and}\quad  \osc_U \tef:=\sup_{x,y\in U}|\tef(x)-\tef(y)|\,.
\end{equation*}
By `$\cdot$' we denote the scalar product of two vectors, i.e.  for ${\texi}=(\xi_1,\dots,\xi_m)\in \Rm$ and 
	$\teeta= (\eta_1,\dots,\eta_m)\in \Rm$ we have ${\texi}\cdot{\teeta} = \sum_{i=1}^m \xi_i \eta_i$, and by `$:$' the Frobenius product of the second-order tensors, i.e. for $\texi=[\xi_{j}^\alpha]_{j=1,\dots,n,\, \alpha=1,\dots,m}$ and $\teeta=[\eta_{j}^\alpha]_{j=1,\dots,n,\, \alpha=1,\dots,m}$ we have $\texi: \teeta=\sum_{\alpha=1}^m \sum_{j=1}^n \xi_{j}^\alpha \eta_{j}^\alpha.$
	By `{$\otimes$}' we denote the tensor product of two vectors, i.e for ${{\texi}}=(\xi_1,\dots,\xi_k)\in {\R^k}$  and 
	${{\teeta}}= (\eta_1,\dots,\eta_l)\in \R^l$ we have $\texi\otimes\teeta:=[\xi_i\eta_j]_{i=1,\dots,k,\,j=1,\dots,l},$ that is \[{\texi}\otimes{\teeta}:=\begin{pmatrix}
  \xi_{1}\eta_{1} & \xi_{1}\eta_{2} & \cdots & \xi_{1}\eta_{l} \\
  \xi_{2}\eta_{1} & \xi_{2}\eta_{2} & \cdots & \xi_{2}\eta_{l} \\
  \vdots  & \vdots  &  & \vdots  \\
  \xi_{k}\eta_{1} & \xi_{k}\eta_{2} & \cdots & \xi_{k}\eta_{l} 
 \end{pmatrix}\in \R^{k\times l}\,.\]

We define
the projection $\P:\Rm\to\R^{m\times m}$  as\begin{equation}
    \label{P}\P(\texi):=\frac{\texi\otimes\texi}{|\texi|^2}\,,
\end{equation}
and the vectorial truncation operator $T_k:\Rm\to\Rm$ as
\begin{equation}\label{Tk}T_k(\texi):=\min\left\{1,\frac{k}{|\texi|}\right\}\texi\,.
\end{equation}
Then, of course, $D T_k:\Rm\to\R^{m\times m}$ is given by\begin{equation}
    \label{DTk}DT_k(\texi)=\begin{cases}
    \id&\text{if }\ |\texi|\leq k\,,\\
    \frac{k}{|\texi|}\left(\id-\P(\texi)\right)&\text{if }\ |\texi|> k\,.
    \end{cases}
\end{equation}

\subsection{On the vector field}
We note that for $\opA$ defined by~\eqref{opA:def} it holds that
\begin{equation*}
    \opA(\texi)
    = \frac{g(|\texi|)}{|\texi|} \xi_{i}^{\alpha} e^{\alpha} \otimes e_{i}\,, \quad \texi \in \rnm\,.
\end{equation*}
The differential of $\opA$ is given by
\begin{equation*}
    \partial \opA(\texi) = \frac{g(|\texi|)}{|\texi|} \mathcal{L}(\texi)\,,
\end{equation*}
where $\cL: \rnm \to \rnm$ is defined as
\begin{equation} \label{eq:L}
    \cL(\texi) := \left( \delta_{\alpha \beta} \delta_{ij} + \frac{|\texi|g'(|\texi|)-g(|\texi|)}{g(|\texi|)} \frac{\xi_{i}^{\alpha} \xi_{j}^{\beta}}{|\texi|^{2}} \right) (e^{\alpha} \otimes e_{i}) \otimes (e^{\beta} \otimes e_{j})\,.
\end{equation}
Note that the vector field $\cL$ satisfies
\begin{equation}\label{L-basic-prop}
    |\teze|^{2} \leq (\cL(\texi): \teze): \teze \qquad\text{and}\qquad |\cL(\texi): \teze| \leq \frac{|\texi| g'(|\texi|)}{g(|\texi|)} |\teze|\,
\end{equation}
for every $\texi, \teze \in \rnm$. In particular, as we assume \eqref{a:p-1}, the second condition above implies
\begin{equation}\label{eq:L-upper}
    |\cL(\texi): \teze| \leq (q-1) |\teze|\qquad \text{for every $\texi, \teze \in \rnm$}\,.
\end{equation}

The second differential of $\opA$ is given by
\begin{align*}
    \partial^{2}\opA(\texi)
    &= \frac{|\texi| g'(|\texi|) - g(|\texi|)}{|\texi|^{2}} \left( \frac{\xi_{i}^{\alpha}}{|\texi|} \delta_{\beta \gamma} \delta_{jk} + \frac{\xi_{j}^{\beta}}{|\texi|} \delta_{\gamma\alpha} \delta_{ki} + \frac{\xi_{k}^{\gamma}}{|\texi|} \delta_{\alpha\beta} \delta_{ij} \right) \\
    &\quad + \frac{|\texi|^{2} g''(|\texi|) - 3|\texi| g'(|\texi|) + 3g(|\texi|)}{|\texi|^{2}} \frac{\xi_{i}^{\alpha}\xi_{j}^{\beta} \xi_{k}^{\gamma}}{|\texi|^{3}}\,.
\end{align*}
By the assumptions \eqref{a:p-1} and \eqref{a:p-2}, we have
\begin{equation}\label{partial-2-A-bound}
    |\partial^{2}\opA(\texi)| \leq c g''(|\texi|)\,.
\end{equation}

\subsection{Orlicz spaces}

 Basic reference for this section is~\cite{adams-fournier}, where the theory of Orlicz spaces is presented for scalar functions. The theory can be extended to $\Rm$-valued functions by obvious modifications.
 
We study solutions to PDEs in Orlicz--Sobolev spaces equipped with a~modular function $G$ satisfying Assumption~\ref{ass:G}. Let us define a modular \begin{equation*}
\vr_{G,\Omega}(|\texi|)=\int_{\Omega} G(|\texi|)\dx\,.
\end{equation*}
We define ${L}^G(\Omega,\Rm)$ to be the space of measurable functions endowed with the Luxemburg norm 
\begin{equation*}
\|\tef\|_{L^G(\Omega, \Rm)}=\inf\left\{\lambda>0:\ \ \vr_{G,\Omega}\left( \tfrac{1}{\lambda} |\tef|\right)\leq 1\right\}\,.
\end{equation*}
 We define the Orlicz--Sobolev space  $W^{1,G}(\Omega)$ by
\begin{equation*} 
W^{1,G}(\Omega,\Rm)=\big\{\tef\in {W^{1,1}(\Omega,\Rm)}:\ \ |\tef|\,,|D{\tef}|\in L^G(\Omega,\Rm)\big\}\,,
\end{equation*}
where the gradient is understood in the distributional sense, endowed with the norm
\begin{equation*}
    \|\tef\|_{W^{1,G}(\Omega,\Rm)}=\inf\bigg\{\lambda>0 :\ \    \vr_{G,\Omega}\left( \tfrac{1}{\lambda} |\tef|\right)+ \vr_{G,\Omega}\left( \tfrac{1}{\lambda} |D{\tef}|\right)\leq 1\bigg\} \, .
\end{equation*}
By $W_0^{1,G}(\Omega,\Rm)$ we denote the closure of $C_c^\infty(\Omega,\Rm)$ under the above norm. Assumption~\ref{ass:G} implies that the Orlicz--Sobolev space $W^{1,G}(\Omega,\Rm)$ is separable and reflexive. Moreover, one can apply arguments of \cite{Gossez} to infer density of  smooth functions  in $W^{1,G}(\Omega,\Rm)$.

\subsection{Solutions to problems with regular data}

Although our main results are stated for SOLA to problems with measure data, we will first establish most auxiliary results in the framework of weak solutions and later run an approximation argument. For this purpose we provide the following definition.

\begin{defi}
Let $\temu \in (W^{1,G}_{0}(\Omega,\Rm))'$. A map $\teu\in W^{1,G}(\Omega,\Rm)$ is called a \emph{weak solution} to~\eqref{eq:mu} if
\begin{equation*}
    \int_\Omega \opA(\teDu): \teDvp\dx=\int_\Omega\tevp\,\mathrm{d}\temu \quad\text{for every }\ \tevp\in W^{1,G}_{0}(\Omega,\Rm)\,.
\end{equation*}
In particular, we call $\teu$ a \emph{$\opA$-harmonic map} if it is a continuous weak solution to~\eqref{eq:mu} with $\temu \equiv 0$.
\end{defi}

\begin{rem}[Existence and uniqueness of weak solutions] \rm For $\temu\in (W^{1,G}_{0}(\Omega,\Rm))',$ due to the strict monotonicity of the operator $\opA$, there exists a unique weak solution to~\eqref{eq:mu} for boundary data belonging to a suitable function space, see~\cite[Section~3.1]{KiSt}.
\label{rem:weak-sol}
\end{rem}

\begin{rem} \rm
    By Campanato's characterization \cite[Theorem~2.9]{giusti} one can infer H\"older continuity $C^{0,\gamma}(\Omega,\Rm)$ of $\opA$-harmonic maps with any exponent $\gamma\in (0,1)$ from \cite[Proposition~3.13]{CYZG}.
\end{rem}

\section{Comparison results} \label{sec:comparison}

The goal of this section is to establish comparison estimates for weak solutions to \eqref{eq:mu}, which are the central ingredients in the proof of the excess decay estimates. We distinguish between the degenerate case (if the size of the gradient average is smaller than the excess functional) and nondegenerate case (otherwise). The main results of this section are Propositions~\ref{prop-5.1}~and~\ref{prop-7.1}.

\subsection{Degenerate approximation}

The following is the main result of this subsection.

\begin{prop}\label{prop-5.1}
    Under Assumption~\ref{ass:G} suppose $\teu\in W^{1,G}(B_r,\Rm)$ is a weak solution to~\eqref{eq:mu}  in $B_r=B_r(x_0)$ with $\temu\in C^\infty(B_r,\Rm)$.  Let $s\in [0,\sa)$ be admissible in Theorem~\ref{theo:Ah-approx} and let $\varepsilon, \theta \in (0,1)$. Then there exists a positive constant $c_{\rm d}=c_{\rm d}(\data,{s},\ve,\theta)$ such that if\begin{equation}
    \label{degeneracy-condition}
    \theta|(\teDu)_{B_r}|\leq \mean{B_r}|\teDu-(\teDu)_{B_r}|\dx\,,
\end{equation}
then there exists an $\opA$-harmonic map $\tev$ in $B_{r/2}$ satisfying\begin{equation*}
    \big({g}^{1+s}\big)^{-1}\left(\mean{B_{r/2}}{g}^{1+s}(|\teDu-\teDv|)\dx\right)\leq {\ve}\,\mean{B_r}|\teDu-(\teDu)_{B_r}|\dx+c_{\rm d} \,g^{-1}\left(\frac{|\temu|(B_r)}{r^{n-1}}\right)\,.
\end{equation*}
\end{prop}

Let us recall a direct consequence of the measure data $\opA$-harmonic approximation lemma provided in \cite[Theorem~4.1]{CYZG}.

\begin{theo}\label{theo:Ah-approx} Let Assumption~\ref{ass:G} be satisfied. Let $\ve>0$, $B_r=B_r(x_0)$, and suppose that there exists $M \ge 1$ such that $\teu\in W^{1,G}(B_r,\Rm)$ satisfies\begin{equation}
    \label{u-male}\mean{B_r}|\teu|\dx\leq Mr\,.
\end{equation}
Then there exists a constant $\sa = \sa(p,q, n) > 0$ such that for all $s \in [0, \sa)$ and some $\delta=\delta(\data,s,M,\ve)\in(0,1]$ the following holds. 

If $\teu$ is almost $\opA$-harmonic in a sense that for every $\tevp\in W^{1,G}_0(B_r,\Rm)\cap L^\infty(B_r,\Rm)$ it holds 
\begin{equation}
    \label{Du-male}\left|\mean{B_r}  
    {\opA(\teDu)}:\teDvp \dx\right|\leq\frac{\delta}{r}\|\tevp\|_{L^\infty(B_r,\Rm)}\,,
\end{equation} 
then there exists an $\opA$-harmonic map $\tev\in W^{1,G}(B_{r/2},\Rm)$ satisfying
    \begin{equation*}
    \mean{B_{r/2}}g^{1+s}(|\teDu-\teDv|)\dx
    \leq \ve\,.
\end{equation*}
\end{theo}
We note that in~\cite{CYZG} the above result is stated for a function growing faster than $g^{1+s}$, but the recalled formulation is enough for our proof.

In the sequel we will make use of the fact that $\opA$-harmonic maps are Lipschitz continuous. We infer this from \cite[Lemma~5.8]{DSV1} and a standard interpolation technique, see \cite[Chapter 10]{giusti}.
\begin{lem}\label{lem:Lip} 
Suppose Assumption~\ref{ass:G} is satisfied and $\tev\in W^{1,G}(\Omega,\Rm)$ is $\opA$-harmonic in $\Omega$. Then there exists $c=c(\data)>0$ such that
\begin{equation*}
\sup_{B_r} |D\tev| \leq c\, \mean{B_{2r}} |D\tev|\dx\,
\end{equation*}
for any $B_{2r}=B_{2r}(x_0) \Subset \Omega$.
\end{lem}

In order to show the basic properties of $\teDu$ we shall make use of rescaled functions
\begin{equation}\label{eq:gGH}
    \overline{g}(t) := \overline{g}_{\lambda}(t) := \frac{g(\lambda t)}{g(\lambda)} \qquad\text{and}\qquad \overline{G}(t) := \overline{G}_{\lambda}(t) := \int_{0}^{t} \overline{g}(\tau) \dtau = \frac{G(\lambda t)}{\lambda g(\lambda)} \qquad\text{for a given }\ \lambda>0\,.
\end{equation}

\begin{lem}\label{lem-5.2} Under Assumption~\ref{ass:G} suppose $\teu\in W^{1,G}(B_r,\Rm)$ is a weak solution to~\eqref{eq:mu} in $B_r=B_r(x_0)$ with $\temu\in C^\infty(B_r,\Rm)$.   Let $s\in [0,\sa)$ be admissible in Theorem~\ref{theo:Ah-approx} and let $\varepsilon \in (0,1)$. Then there exist $c_{\rm s}=c_{\rm s}(\data,{s},\ve) > 0$ and a map $\tev$, which is $\opA$-harmonic in $B_{r/2}$, such that
\begin{equation}
    \label{comp-est-1}\big(g^{1+s}\big)^{-1}\left(\mean{B_{r/2}}g^{1+s}(|\teDu-\teDv|)\dx\right)\leq \frac{\varepsilon}{r}\mean{B_r}|\teu-(\teu)_{B_r}|\dx+ c_{\rm s} \,g^{-1}\left(\frac{|\temu|(B_r)}{r^{n-1}}\right).
\end{equation}  
\end{lem}
\begin{proof}
Let $\tilde{\varepsilon} > 0$ and fix
\[\lambda:=\frac{1}{r}\mean{B_r}|\teu-(\teu)_{B_r}|\dx+g^{-1}\left(\frac{\delta}{|B_1|} \frac{|\temu|(B_r)}{r^{n-1}}\right),\]
where and $\delta=\delta(\data,{s},\tilde{\ve})$ is a constant from Theorem~\ref{theo:Ah-approx} with $M=1$. If $\lambda=0,$ then $\teu$ is constant and $\tev=\teu$. Otherwise we argue by scaling
\begin{equation}\label{eq-scaling}
\tebu:=\frac{\teu-(\teu)_{B_r}}{\lambda}\,,\qquad \btemu:=\frac{\temu}{g(\lambda)}\,,\qquad \bopA(\texi):=\frac{\opA(\lambda\texi)}{g(\lambda)}= \frac{\overline{g}(|\texi|)}{|\texi|}\,\texi\,,
\end{equation}
where $\overline{g}$ is defined as in \eqref{eq:gGH}. Then we have
\[\left|\mean{B_r} \bopA(\teDbu):\teDvp\dx\right|\leq \frac{\|\tevp\|_{L^\infty(B_r,\Rm)}|\temu|(B_r)}{g(\lambda) |B_r|}\leq\frac{\delta}{r}\|\tevp\|_{L^\infty(B_r,\Rm)}\,.\]
By definition of $\tebu$ and $\lambda$ we notice that
\[\mean{B_r}|\tebu|\dx\leq r\,.\]
Therefore, by Theorem~\ref{theo:Ah-approx} applied to $\tebu$ there exists an $\bopA$-harmonic map $\tebv$ in $B_{r/2}$ such that 
\[\mean{B_{r/2}} \overline{g}^{1+s}(|\teDbu-\teDbv|)\dx\leq \tilde{\ve}\,.\] 
By scaling back with $\tev=\lambda\tebv$, which is $\opA$-harmonic, and setting $\tilde{\varepsilon}=\varepsilon^{(q-1)(1+s)}$, we obtain
\begin{equation*}
\mean{B_{r/2}} g^{1+s}(|\teDu-\teDv|)\dx\leq \varepsilon^{(q-1)(1+s)} g^{1+s}(\lambda) \leq g^{1+s}(\varepsilon \lambda)\,,
\end{equation*}
which concludes \eqref{comp-est-1}.
\end{proof}

We are in a position to prove the approximation of a solution $\teu$ by an $\opA$-harmonic function $\tev$.

\begin{proof}[Proof of Proposition~\ref{prop-5.1}]
We have from \eqref{degeneracy-condition} that
\begin{equation*}
\mean{B_{r}} |D\teu| \dx \leq \mean{B_{r}} |D\teu - (D\teu)_{B_{r}}| \dx + |(D\teu)_{B_{r}}| \leq \frac{1+\theta}{\theta} \mean{B_{r}} |D\teu - (D\teu)_{B_{r}}| \dx\,.
\end{equation*}
By Poincar\'e's inequality, we have
\begin{equation} \label{eq:PI}
\mean{B_{r}} |\teu-(\teu)_{B_{r}}| \dx \leq C r \mean{B_{r}} |D\teu| \dx \leq C\frac{1+\theta}{\theta} r \mean{B_{r}} |D\teu - (D\teu)_{B_{r}}| \dx\,.
\end{equation}
{By Lemma~\ref{lem-5.2} with $\tilde{\varepsilon} = \frac{\theta}{C(1+\theta)}\varepsilon$, there exists an $\opA$-harmonic map $\tev$ in $B_{r/2}$ such that
\begin{equation*}
\big(g^{1+s}\big)^{-1}\left(\mean{B_{r/2}}g^{1+s}(|\teDu-\teDv|)\dx\right)\leq \frac{\tilde{\varepsilon}}{r}\mean{B_r}|\teu-(\teu)_{B_r}|\dx+ c_{\rm s} \,g^{-1}\left(\frac{|\temu|(B_r)}{r^{n-1}}\right),
\end{equation*}
where $c_{\rm s}=c_{\rm s}(\data, s, \tilde{\varepsilon})$. Thus, it follows from \eqref{eq:PI} that \begin{equation*}
\big({g}^{1+s}\big)^{-1}\left(\mean{B_{r/2}}{g}^{1+s}(|\teDu-\teDv|)\dx\right)\leq {\ve}\,\mean{B_r}|\teDu-(\teDu)_{B_r}|\dx+c_{\rm d} \,g^{-1}\left(\frac{|\temu|(B_r)}{r^{n-1}}\right)\,,
\end{equation*}
where $c_{\rm d}$ is a constant depending only on $\data, s, \varepsilon$ and $\theta$.}
\end{proof}

\subsection{Nondegenerate linearization}
The following is the main result of this subsection.

\begin{prop}
    \label{prop-7.1} 
    Under Assumption~\ref{ass:G} suppose $\teu\in W^{1,G}(B_r,\Rm)$ is a weak solution to~\eqref{eq:mu}  in $B_r=B_r(x_0)$ with $\temu\in C^\infty(B_r,\Rm)$, and assume that $(\teDu)_{B_r}\neq 0$. 
Let $\cL$ be defined as in~\eqref{eq:L}. 
Then for every $\ve\in (0,1]$ there exists a constant $\theta_{\rm nd}=\theta_{\rm nd}(\data,\ve)\in(0,1)$   such that if both
    \begin{equation*}
        \mean{B_r}|\teDu-(\teDu)_{B_r}|\dx \leq 
         {\theta_{\rm nd}}|(\teDu)_{B_r}|
    \end{equation*}
    and
    \begin{equation}
        \label{nice-measure}
        \frac{|\temu|(B_r)}{r^{n-1}} \leq \theta_{\rm nd}\frac{g(|(\teDu)_{B_r}|)}{|(\teDu)_{B_r}|}\mean{B_r}|\teDu-(\teDu)_{B_r}|\dx
    \end{equation}
    hold, then there exists a map $\teh\in W^{1,2}(B_{r/4},\Rm)$ which is $A$-harmonic in $B_{r/4}$, that is a solution to
    \begin{equation*}
        -\tedv (A:D\teh)=0\ \ \text{weakly in $B_{r/4}\ $ for $\ A:= \cL((\teDu)_{B_r})$}\,,
   \end{equation*}
    such that
    \begin{equation*}
        \mean{B_{r/4}}|\teDu-D\teh|\dx \leq {\ve}\,\mean{B_r}|\teDu-(\teDu)_{B_r}|\dx\,.
    \end{equation*}
\end{prop}
   
Let us fix some notation. We will deal with $\tebu\in W^{1,G}(B_1,\Rm)$ being a weak solution to\begin{equation}
    \label{bar-system}-\tedv \bopA (\teDbu)=\btemu\in C^\infty(B_1)\,,
\end{equation}
which is such that \begin{equation}
    \label{normalized-bar-u}|(\teDbu)_{B_1}|=1\qquad\text{and}\qquad|(\tebu)_{B_1}|=0\,. 
\end{equation}
We define $\teell\equiv (\ell^\alpha)_{1\leq\alpha\leq m}$ and $\tebv$ as\begin{equation}\label{def:ell}
    \ell^\alpha(x):=\langle (\teDbu^\alpha)_{B_1},x\rangle\qquad\text{and}\qquad \tebv:=\tebu-\teell\,.
\end{equation}
Then
\begin{equation}
    \label{normalized-bar-v}|(\teDbv)_{B_1}|=0\qquad\text{and}\qquad|(\tebv)_{B_1}|=0\,. 
\end{equation}
Assume further that\begin{equation}
    \label{normalized-bar-v-2}
    \mean{B_1}|\teDbv|\dx\leq 1\qquad\text{and}\qquad|\btemu|(B_1)\leq 1\,,
\end{equation}
then~\eqref{normalized-bar-u} and~\eqref{normalized-bar-v-2} imply that\begin{equation}
    \label{Dbu-small}\mean{B_1}|\teDbu|\dx\leq \mean{B_1}|\teDbv|\dx+|(\teDbu)_{B_1}|\leq 2\,.
\end{equation}
We observe that for $\tau>1$ it holds\begin{equation}
    \begin{cases}\label{trivial-inclusions}
        B_1\cap \{|\tebv|>\tau\}\subset B_1\cap\{|\tebu|>\tau-1\}\,,\\
        B_1\cap \{|\tebu|>\tau\}\subset B_1\cap\{|\tebv|>\tau-1\}\,,\\
        B_1\cap \{|\teDbv|>\tau\}\subset B_1\cap\{|\teDbu|>\tau-1\}\,,\\
        B_1\cap \{|\teDbu|>\tau\}\subset B_1\cap\{|\teDbv|>\tau-1\}\,.
    \end{cases}
\end{equation}

\begin{lem}\label{lem-5.1} Under Assumption~\ref{ass:G} suppose $\teu\in W^{1,G}(B_r,\Rm)$ is a weak solution to~\eqref{eq:mu} in $B_r=B_r(x_0)$ with $\temu\in C^\infty(B_r,\Rm)$.   Let $s\in [0,\sa)$ be admissible in Theorem~\ref{theo:Ah-approx}. Then there exists a positive constant $c_{\rm h}=c_{\rm h}(\data,s)$ such that\begin{equation*}
    \big(g^{1+s}\big)^{-1}\left(\mean{B_{r/2}}g^{1+s}(|\teDu|)\dx\right)\leq c_{\rm h}\left[\mean{B_r}|\teDu|\dx+g^{-1}\left(\frac{|\temu|(B_r)}{r^{n-1}}\right)\right]\,.
\end{equation*}
\end{lem}
    
\begin{proof}
Let us fix a constant $\varepsilon >0$ and let $\delta \in (0, 1]$ be the constant given in Theorem~\ref{theo:Ah-approx} with $M=1$. Set
\begin{equation*}
\lambda := \frac{1}{\varepsilon} \mean{B_{r}} |D\teu| \dx + g^{-1} \left( \frac{2^{n-1}r}{\delta} \frac{|\temu|(B_{r})}{|B_{r}|} \right)\,
\end{equation*}
and use a scaling argument with \eqref{eq-scaling}. To apply Theorem~\ref{theo:Ah-approx} to $\tebu$ in $B_{r/2}(y) \subset B_{r}$ with $\overline{g}$ and $\overline{G}$ given by \eqref{eq:gGH}, we check \eqref{u-male} and \eqref{Du-male} for $\tebu$. Indeed, scaling and the definition of $\lambda$ give
\begin{equation} \label{eq:scaling-Du}
\mean{B_{r}} |D\tebu| \dx = \frac{1}{\lambda} \mean{B_{r}} |D\teu| \dx \leq \varepsilon\,,
\end{equation}
and hence by Poincar\'e's inequality
\begin{equation*}
\mean{B_{r/2}(y)} |\tebu| \dx \leq 2^{n} \mean{B_{r}} |\tebu| \dx \leq Cr \mean{B_{r}} |D\tebu| \dx \leq C \varepsilon r \leq \frac{r}{2}
\end{equation*}
for every $y \in B_{r/2}(x_0)$, provided $\varepsilon = 1/(2C)$. Note that the dependence of $\delta$ on $\varepsilon$ is removed at this point. Moreover, since $-{\tedv} \bopA(\teDbu) = \btemu$ in $B_{r} \subset \Omega$, we obtain
\begin{equation*}
\left| \mean{B_{r/2}(y)} \bopA(\teDbu) : D{\tevp} \dx \right| \leq \frac{2^{n} \|\tevp\|_{L^{\infty}(B_{r/2}(y), \Rm)}}{g(\lambda)} \frac{|\btemu|(B_{r})}{|B_{r}|} \leq \frac{\delta}{r/2} \|\tevp\|_{L^{\infty}(B_{r/2}(y), \Rm)}
\end{equation*}
for every $y \in B_{r/2}(x_{0})$ and $\tevp \in W_{0}^{1, \overline{G}}(B_{r/2}(y),\Rm) \cap L^{\infty}(B_{r/2}(y),\Rm)$. Thus, by Theorem~\ref{theo:Ah-approx} there exists an $\mathcal{A}$-harmonic map $\tebv \equiv \tebv_{y} \in W^{1, \overline{G}}(B_{r/4}(y),\Rm)$ such that
\begin{equation} \label{eq:scaling-Du-Dv}
\mean{B_{r/4}(y)} \overline{g}^{1+s} (|D\tebu-D\tebv|) \dx \leq \varepsilon\,.
\end{equation}
The Lipschitz estimate for $D\tebv$ from Lemma~\ref{lem:Lip}, Jensen's inequality,  \eqref{eq:scaling-Du} and \eqref{eq:scaling-Du-Dv} show
\begin{equation*}
\begin{split}
\sup_{B_{r/8}(y)} |D\tebv|
&\leq C \mean{B_{r/4}(y)} |D\tebv| \dx \leq C \big(\overline{g}^{1+s}\big)^{-1} \left( \mean{B_{r/4}(y)} \overline{g}^{1+s}(|D\tebu-D\tebv|) \dx \right) + C \mean{B_{r/4}(y)} |D\tebu| \dx \leq C\,.
\end{split}
\end{equation*}
Thus, we have
\begin{equation*}
\mean{B_{r/8}(y)} \overline{g}^{1+s} (|D\tebu|) \dx \leq C \sup_{B_{r/8}(y)} \overline{g}^{1+s}(|D\tebv|) + C \mean{B_{r/4}(y)} \overline{g}^{1+s} (|D\tebu-D\tebv|) \dx \leq C\,.
\end{equation*}
A simple covering argument shows
\begin{equation*}
\mean{B_{r/2}(x_{0})} \overline{g}^{1+s} (|D\tebu|) \dx \leq C\,.
\end{equation*}
By scaling back to $\teu$, we conclude with
\begin{equation*}
\mean{B_{r/2}(x_{0})} {g}^{1+s} (|D\teu|) \dx \leq C {g}^{1+s}(\lambda) \leq C {g}^{1+s}\left( \mean{B_{r}} |D\teu| \dx + g^{-1} \left( \frac{|\temu|(B_{r})}{r^{n-1}} \right)\right)\,,
\end{equation*} 
where $C$ depends only on $\data$ and $s$.
\end{proof}

The following lemma is a reverse H\"older inequality on level sets.

\begin{lem}[Reverse H\"older inequality]
    \label{lem-6.1}  Under Assumption~\ref{ass:G} suppose $\teu\in W^{1,G}(B_r,\Rm)$ is a weak solution to~\eqref{eq:mu} in $B_r=B_r(x_0)$ with $\temu\in C^\infty(B_r,\Rm)$.  Let $s\in [0,\sa)$ be admissible in Theorem~\ref{theo:Ah-approx}. Moreover, assume
\begin{equation}
    \label{smallness}
    20^n c_{\rm h}\left[\mean{B_r}|\teDu|\dx+g^{-1}\left(\frac{|\temu|(B_r)}{r^{n-1}}\right)\right]< t\,,
\end{equation} 
where $c_{\rm h}=c_{\rm h}(\data,s)$ is the constant from Lemma~\ref{lem-5.1}. Then, it holds
\begin{equation*}
    \int_{B_{r/2} \cap \{|\teDu| > t\} } {g}^{1+s}(|\teDu|) \dx \le c_* \frac{{g}^{1+s}(t)}{t} \int_{B_r \cap \{|\teDu| > t/{c_{\ast}}\} } |\teDu| \dx + c_* r {g^s(t)} |\temu|(B_r)\,,
\end{equation*}
where $c_* = c_*(\data, s) > 0$ is a constant.
\end{lem}
\begin{proof}
The proof follows along the lines of the proof of \cite[Lemma 6.1]{KuMi2018}.
    Let $t$ be as in \eqref{smallness} and define the level set \[S_t := B_{r/2} \cap \{|\teDu| > t\}\,.\] First, for any $y\in S_t$ by using Lemma~\ref{lem-5.1} and \eqref{a:p} we obtain for any $y \in S_t$
    \begin{align*}
        \big({g}^{1+s})^{-1}\left(\mean{B_{r/20}(y)}{g}^{1+s}(|\teDu|)\dx\right) &\leq c_{\rm h}\left[\mean{B_{r/10}(y)}|\teDu|\dx+g^{-1}\left(\frac{|\temu|(B_{r/10}(y))}{(r/10)^{n-1}}\right)\right]\\
        &\le 10^n c_{\rm h}\left[\mean{B_{r}}|\teDu|\dx+g^{-1}\left(\frac{|\temu|(B_{r})}{r^{n-1}}\right)\right]< t\,,
    \end{align*}
    where we also used the smallness assumption \eqref{smallness} in the last step. As a consequence, for almost every $y \in S_t$ there exists an exit time radius $10r_y \in (0,r/20)$ such that
    \begin{align}
        \big({g}^{1+s}\big)^{-1}\left( \mean{B_{10 r_y}(y)} {g}^{1+s}(|\teDu|) \dx \right) = t\,, \qquad \max_{\vr \in [10 r_y, r/20]}& \big({g}^{1+s}\big)^{-1}\left( \mean{B_{\vr}(y)} {g}^{1+s}(|\teDu|) \dx \right) \le t\,.
    \label{eq:lem-6.1_help1}
    \end{align} 
    Thus, by Lemma~\ref{lem-5.1}, at least one of the following estimates holds true:
    \begin{align*}
        \frac{t}{2} \le c_{\rm h} \mean{B_{2r_y}(y)} |\teDu| \dx, \qquad \frac{t}{2} \le c_{\rm h} g^{-1}\left(\frac{|\temu|(B_{2r_y}(y))}{(2r_y)^{n-1}}\right)\,.
    \end{align*}
By the same argument as in \cite{KuMi2018}, this implies that we have:
 \begin{align}
 \label{eq:lem-6.1_help2}
     |B_{2 r_y}(y)| \le \frac{4 c_{\rm h}}{t} \int_{B_{2 r_y}(y) \cap \{|\teDu| > t/(4 c_{\rm h})\} } |\teDu| \dx + \frac{c_1 r_y |\temu|(B_{2r_y}(y))}{g \left( \frac{t}{2 c_{\rm h}} \right)}
 \end{align}   
 for some constant $c_1 = c_1(n) > 0$. Let us apply the Vitali covering lemma to extract a countable family of disjoint balls $\{B_{2r_{y_j}}(y_j)\}_j$ satisfying 
 \begin{align*}
     S_t \subset \left( \bigcup_{j} B_{10r_{y_j}}(y_j) \right) \cup \text{ negligible set}\,.
 \end{align*}
 We will now apply \eqref{eq:lem-6.1_help1} and \eqref{eq:lem-6.1_help2} to each ball of the family to estimate for any $j$: 
 \begin{align*}
     \int_{B_{10r_{y_j}}(y_j) }  {g}^{1+s}(|\teDu|) \dx &\le 5^n {g}^{1+s}(t) |B_{2r_{y_j}}(y_j)| \\
     &\le c_2 \frac{{g}^{1+s}(t)}{t} \int_{B_{2 r_{y_j}}(y_j) \cap \{|\teDu| > t/(4 c_{\rm h})\} } |\teDu| \dx + c_2 r {g^s(t)} |\temu|(B_{2r_{y_j}}(y_j))\,,
 \end{align*}
 where $c_2 = c_2(\data, s) > 0$ is a constant. The desired result follows by summing the previous estimate over $j$.
\end{proof}

\begin{lem}
    \label{lem-7.1} If $\tebv$ and $\btemu$ are as in \eqref{bar-system}--\eqref{normalized-bar-v-2}, $\tau>3$, and $\vk\geq 1$, then there exists $c=c(\data,\tau,{\vk})>0$ such that
    \begin{equation*}
        \int_{B_{7/8}\cap\{3<|\tebv|\leq \tau\}} \left(|\teDbv|^2+G(|\teDbv|)\right)\dx\leq c\left(\int_{B_1}  |\teDbv|^{\vk}\dx+|\btemu|(B_1)\right)\,.
    \end{equation*}
\end{lem}
\begin{proof}
    Following \cite{KuMi2018} we fix $m>\frac{1}{2(\tau-1)}$ and define\begin{equation*}
        \xi(t):=\min\{1,2\tau/t\}\,\min\{m(t-2)_+,1\}\,.
    \end{equation*}
{We will choose $m$ large.} Let us take $\tebu$ satisfying \eqref{bar-system} and recall the form of the gradient of the vectorial truncation~\eqref{DTk}. Then we can directly compute that\begin{equation}
    \label{gradxiuu}
    D(\xi(|\tebu|)\tebu)=\begin{cases}
        0\,, & |\tebu|\leq 2\,,\\
        m((|\tebu|-2)\id+|\tebu|\P)\teDbu\,, & 2<|\tebu|\leq 2+1/m\,,\\
        \teDbu\,, & 2+1/m<|\tebu|<2\tau\,,\\
        \tfrac{2\tau}{|\tebu|}(\id-\P)\teDbu\,, &|\tebu|\geq 2\tau\,,
    \end{cases}
\end{equation}
    where $\P:=\P(\tebu)$ from in~\eqref{P}. Since $\teDbu:[(\id-\P)\teDbu]\geq 0$, we have that
    \begin{align}
        \teDbu:D(\xi(|\tebu|)\tebu)\mathds{1}_{\{|\tebu|\geq 2\tau\}} &=\tfrac{2\tau}{|\tebu|}\teDbu:[(\id-\P)\teDbu] \mathds{1}_{\{|\tebu|\geq 2\tau\}}\geq 0, \nonumber\\
        \label{on-the-annulus}
        \teDbu:D(\xi(|\tebu|)\tebu) &\geq |\teDbu|^2 \mathds{1}_{\{2+1/m<|\tebu|< 2\tau\}}\,.
    \end{align}
We note that \begin{equation}\label{tensvp-for-testing}
         \tevp:=\phi \,\xi(|\tebu|)\tebu\in W^{1,G}_0(B_1,\Rm)\,,
    \end{equation}
    for $\phi\in C_0^\infty(B_{15/16})$, $0\leq \phi\leq 1$, $\phi\equiv 1$ in $B_{7/8}$, and $|D\phi|\leq 2^8$,
    is an admissible test function for the system~\eqref{bar-system}. By collecting~\eqref{gradxiuu}, \eqref{on-the-annulus}, \eqref{tensvp-for-testing}, \eqref{a:p}, and letting $m\to\infty$, we get\begin{equation}\label{energy-on-annulus}
 p\int_{B_1\cap \{2<|\tebu|<2\tau\}} G(|\teDbu|)\,\phi\dx\leq 2\tau\int_{B_1\cap \{|\tebu|>2\}}g(|\teDbu|)|D\phi|\dx+2\tau|\btemu|(B_1)\,.
    \end{equation}
In order to estimate the right-hand side above we notice that
    \begin{align}
        \int_{B_1\cap \{|\tebu|>2\}}&g(|\teDbu|)|D\phi|\dx\leq
     \|D\phi\|_{L^\infty}\int_{B_{15/16}\cap \{|\teDbu|>H\}} g(|\teDbu|)\dx+g(H){|B_{15/16}\cap \{|\tebu|>2\}|}\,,
    \label{est-1}
    \end{align}
    for $H>1$ to be chosen. In order to estimate the first integral on the right-hand side we will make use of Lemma~\ref{lem-6.1} {with $s=0$}. With this aim we cover $B_{15/16}$ with a finite covering $\{\cB_i\}_{i=1}^{N}$ of open balls with radius $2^{-7}$ where $N=N(n)$ and $\cB_i\cap B_{15/16}$ is not empty. We note that $2\cB_i\subset B_1$ for each $i$. Moreover, on each ball it holds
    \begin{align*}
    {20^n} c_{\rm h}&\left[\mean{2\cB_i}|\teDbu|\dx+g^{-1}\left(\frac{|\btemu|(2\cB_i)}{2^{-6(n-1)}}\right)\right]\leq c \mean{B_1}|\teDbu|\dx+ c g^{-1}(|\btemu|(B_1))\leq c(\data)\leq H\,,
    \end{align*}
    where we can choose $H=H(\data)$ so large that $H\geq 2c_*$ and $c_*,c_{\rm h}$ are constants from Lemma~\ref{lem-6.1} that depend on $\data$ only. We applied~\eqref{normalized-bar-v-2} and~\eqref{Dbu-small} in the above. Consequently,  assumptions of Lemma~\ref{lem-6.1} are satisfied on each $\cB_i$ and we can infer that\begin{align*}
        \int_{B_{15/16}\cap \{|\teDbu|>H\}} g(|\teDbu|)\dx&\leq\sum_{i=1}^N
        \int_{\cB_{i}\cap \{|\teDbu|>H\}} g(|\teDbu|)\dx\\
        &\leq c \sum_{i=1}^N \int_{2\cB_i\cap \{|\teDbu|>H/c_*\}}|\teDbu|\dx+c\sum_{i=1}^N|\btemu|(2\cB_i)\,\\
        &\leq c \int_{B_1\cap \{|\teDbu|>H/c_*\}}|\teDbu|\dx+c\sum_{i=1}^N|\btemu|(B_1)\,,
    \end{align*}
    for $c=c(\data)$. Recalling that $H/c_*\geq 2$  we get that 
    \begin{equation*}
        \int_{B_{15/16}\cap \{|\teDbu|>H\}} g(|\teDbu|)\dx\leq c\,\int_{B_1\cap\{|\teDbu|>2\}}|\teDbu|^{\vk}\dx+c|\btemu|(B_1)\,,
    \end{equation*}
    for $c=c(\data)$ and any $\vk\geq 1$.     Let us estimate the second term on the right-hand side of~\eqref{est-1}. By~\eqref{trivial-inclusions} we know that\[ B_1\cap \{|\tebu|>2\}\subset B_1\cap \{|\tebv|>1\}\,.\] In turn, recalling~\eqref{normalized-bar-v}, Poincar\'e inequality leads to\begin{equation*}
        |B_{15/16}\cap \{|\tebu|>2\}|\leq |B_1\cap \{|\tebv|>1\}|\leq\int_{B_1}|\tebv|^{\vk}\dx\leq\int_{B_1}|\teDbv|^{\vk}\dx\,. 
    \end{equation*}
    Consequently, \eqref{est-1} implies that
    \begin{align*}
        \int_{B_1\cap \{|\tebu|>2\}}g(|\teDbu|)|D\phi|\dx&\leq
     c\,\int_{B_1\cap\{|\teDbu|>2\}}|\teDbu|^{\vk}\dx+c|\btemu|(B_1)+c\,\int_{B_1}|\teDbv|^{\vk}\dx\,\\
        &\leq\,c\,\int_{B_1}|\teDbv|^{\vk}\dx+c|\btemu|(B_1)\,,
    \end{align*}
    where $c=c(\data)$. In the last line we use that $|\teDbu|\leq 2|\teDbv|$ on ${B_1\cap\{|\teDbu|>2\}}$. Applying the content of the last display to~\eqref{energy-on-annulus} we obtain that
    \[ \int_{B_1\cap \{2<|\tebu|<2\tau\}} G(|\teDbu|)\,\phi\dx\leq c\int_{B_1}|\teDbv|^{\vk}\dx+c|\btemu|(B_1)\,,\]
    for $c=c(\data,\tau)\,.$ With the help of~\eqref{trivial-inclusions} and the previous display one can deduce now that
    \begin{align*}
        &\int_{B_{7/8}\cap\{3<|\tebv|\leq \tau\}} \left(|\teDbv|^2+G(|\teDbv|)\right)\dx\\
        &\qquad\leq  \int_{B_{7/8}\cap\{3<|\tebv|\leq \tau\}\cap\{|\teDbu|< 1\}} \left(|\teDbv|^2+G(|\teDbv|)\right)\dx+ \int_{B_{7/8}\cap\{3<|\tebv|\leq \tau\}\cap\{|\teDbu|\geq 1\}} \left(|\teDbv|^2+G(|\teDbv|)\right)\dx\\
        &\qquad\leq  c|B_{7/8}\cap\{3<|\tebv|\leq \tau\}|+ c\int_{B_{7/8}\cap\{2<|\tebu|\leq 2\tau\}} {G(|\teDbu|)}\dx\\
        &\qquad\leq  c\int_{B_1}|\tev|^{\vk}\dx+ c\int_{B_1}|\teDbv|^{\vk}\dx+c|\btemu|(B_1)\,\\
        &\qquad\leq   c\int_{B_1}|\teDbv|^{\vk}\dx+c|\btemu|(B_1)\,,
    \end{align*}
    which completes the proof.
\end{proof}

\begin{lem}
    \label{lem-7.2} Let $\tebv$ and $\btemu$ be as in \eqref{bar-system}--\eqref{normalized-bar-v-2}. Let $\tau>3$. There exist $\gamma_0 \in (0,1)$, $\varkappa_0 \in (1,2)$ depending only on $q$ such that for every $\gamma \in (0,\gamma_0)$ and $\varkappa \in [1,\varkappa_0)$ it holds
    \begin{equation*}
        \int_{B_{1/2}\cap\{|\tebv|\leq \tau\}} \frac{|\teDbv|^2+G(|\teDbv|)}{|\tebv|^\gamma }\dx\leq c\left(\int_{B_1} |\teDbv|^{\vk}\dx+|\btemu|(B_1)\right)\,,
    \end{equation*} 
where $c=c(\data,\tau,\gamma,\varkappa)>0$.
\end{lem}

\begin{proof}  For $\tebu$ satisfying \eqref{bar-system} we consider a system
\begin{equation}
    \label{7.22}-\tedv(\bopA(\teDbu) - \bopA(\teell))=\btemu\,.
\end{equation}
Note that $\tebv$ solves a linearized version of \eqref{7.22}, that is \begin{equation*}
        -\tedv(B(x):\teDbv)=\btemu
    \end{equation*}
with \begin{equation*}
        B(x):=\int_0^1  \frac{g(|D\teell+s\teDbv|)}{|D\teell+s\teDbv|} \cL(D\teell+s\teDbv)\ds\,,
    \end{equation*}
    where $\cL$ is given by~\eqref{eq:L}. 
Then by~\eqref{L-basic-prop} and \eqref{a:p-1} we have
    \begin{equation}\label{B-lower-bound}
       \int_0^1 \frac{g(|D\teell+s\teDbv|)}{|D\teell+s\teDbv|}\ds\,|\teDbv|^2 \leq (B(x):\teDbv):\teDbv \leq (q-1) \int_0^1 \frac{g(|D\teell+s\teDbv|)}{|D\teell+s\teDbv|} \ds\,|\teDbv|^2  \,.
    \end{equation}
By Lemma~\ref{lem-g'} we get that \[\int_0^1\frac{g(|D\teell+s\teDbv|)}{|D\teell+s\teDbv|}\ds\approx \frac{g(|D\teell|+|\teDbv|)}{|D\teell|+|\teDbv|}=\frac{g(|\teDbv|+1)}{|\teDbv|+1}\,.\]
Consequently
\begin{equation}\label{B-trapped}
    \frac{1}{c}|\teDbv|^2\leq \left((B(x):\teDbv):\teDbv\right)\frac{|\teDbv|+1}{g(|\teDbv|+1)}\leq c|\teDbv|^2
\end{equation}
for $c=c(\data) > 0$.  By Lemma~\ref{lem-G} and~\eqref{B-trapped} we see that there holds a pointwise estimate
\begin{equation*}
    |\teDbv|^2+G(|\teDbv|)\leq c (B(x):\teDbv):\teDbv
\end{equation*}
for some $c=c(\data) > 0$. In order to prepare for testing~\eqref{7.22}, we define 
\begin{align}
    \label{Phi}&\text{
$\Phi\in C_0^\infty((-6,6))$ such that $\Phi\equiv 1$ on $[-3,3]$ and $|\Phi'|\leq 1$}\,,\\
    &\text{$\phi\in C_0^\infty(B_{3/4})$ such that $0\leq\phi\leq 1,$ $\phi\equiv 1$ in $B_{1/2}$ and $|D\phi|\leq 128$}\,.\nonumber\end{align} 
    
Recall the vectorial truncation operator $T_t$ given by~\eqref{Tk}. We test~\eqref{7.22} with\[\tevp:=\phi\,\Phi(|\tebv|/\tau)T_t(\tebv)\in W^{1,G}_0(B_1,\rn)\,,\]
where $t\in (0,6\tau]$, to get\begin{align*}
    \int_{B_1}\tevp\,\mathrm{d}\btemu=&\int_{B_1\cap\{|\tev|\leq t\}}\left((B(x):\teDbv):\teDbv\right)\Phi(|\tebv|/\tau)\phi\dx\\
    &+t\int_{B_1\cap\{|\tev|> t\}}(B(x):\teDbv):[(\id-\P)\teDbv]\frac{\Phi(|\tebv|/\tau)}{|\tebv|}\phi\dx\\
    &+\frac{1}{\tau}\int_{B_1\cap\{|\tebv|\leq t\}}\left((B(x):\teDbv):(\P\teDbv)\right)\Phi'(|\tebv|/\tau)|\tebv|\phi\dx\\
    &+\frac{t}{\tau}\int_{B_1\cap\{|\tebv|> t\}}\left((B(x):\teDbv):(\P\teDbv)\right)\Phi'(|\tebv|/\tau)\phi\dx\\
    &+ \int_{B_1\cap\{|\tebv|\leq t\}}\left((B(x):\teDbv):(\tebv\otimes D\phi)\right)\Phi(|\tebv|/\tau)\dx\\
    &+ t\int_{B_1\cap\{|\tebv|> t\}}\left((B(x):\teDbv):\left(\frac{\tebv}{|\tebv|}\otimes D\phi\right)\right)\Phi(|\tebv|/\tau)\dx\,,
\end{align*}
where $\P$ is given by~\eqref{P}.
We observe that
    \begin{align*}
        |(B(x):\teDbv):(\id-\P)\teDbv|&\leq  (q-1)\int_0^1  \frac{g(|D\teell+s\teDbv|)}{|D\teell+s\teDbv|}\ds\,|\teDbv|^2.
    \end{align*}
Indeed, we can infer it from  $|(\id-\P)\teDbv|\leq |\teDbv|$, \eqref{eq:L-upper}, and \eqref{L-basic-prop}. 
Moreover, due to~\eqref{B-lower-bound}, by dividing the previous two estimates, we obtain 
\begin{equation*}
    Q:=\sup _{x\in B_1}\left|\frac{(B(x):\teDbv):(\id-\P)\teDbv}{ (B(x):\teDbv):\teDbv}\right|\leq q-1\,.
\end{equation*}
Let us take $\ve\in (0,\tau/2)$, $K\geq 6$, $\gamma\in(0,1)$ to be chosen later. We multiply the above equality by $t^{-1-\gamma}$, using estimates on $B$ from~\eqref{B-trapped}, and integration over $(\ve,K\tau)$ with respect to $t$, we obtain
\begin{align}
    \nonumber\frac{(K\tau)^{1-\gamma}}{1-\gamma}|\btemu|(B_1)\geq &\int_\ve^{K\tau}t^{-1-\gamma}\int_{B_1\cap\{|\tev|\leq t\}}\left[(B(x):\teDbv):\teDbv\right]\Phi(|\tebv|/\tau)\phi\dx\dt\\
    \nonumber&-Q\int_\ve^{K\tau}t^{-\gamma}\int_{B_1\cap\{|\tev|> t\}}\left[(B(x):\teDbv):\teDbv\right]\frac{\Phi(|\tebv|/\tau)}{|\tebv|}\phi\dx\dt\\
    \nonumber&-c\int_0^{K\tau}\frac{t^{-1-\gamma}}{\tau}\int_{B_1\cap\{3\tau<|\tebv|<6\tau\}\cap\{|\tebv|\leq t\}}\frac{g(|\teDbv|+1)}{|\teDbv|+1}|\teDbv|^2\Phi'(|\tebv|/\tau)|\tebv|\phi\dx\dt\\
    \nonumber&-c\int_0^{K\tau}\frac{t^{-\gamma}}{\tau}\int_{B_1\cap\{3\tau<|\tebv|<6\tau\}\cap\{|\tebv|> t\}}\frac{g(|\teDbv|+1)}{|\teDbv|+1}|\teDbv|^2\Phi'(|\tebv|/\tau)\phi\dx\dt\\
    \nonumber&-c \int_0^{K\tau}t^{-1-\gamma}\int_{B_1\cap\{|\tebv|\leq t\}}\frac{g(|\teDbv|+1)}{|\teDbv|+1}|\teDbv|\, |D\phi|\,|\tebv|\,\Phi(|\tebv|/\tau)\dx\dt\\
    \nonumber&-c\int_0^{K\tau}t^{-\gamma}\int_{B_1\cap\{|\tebv|> t\}}\frac{g(|\teDbv|+1)}{|\teDbv|+1}|\teDbv|\,|D\phi|\,\Phi(|\tebv|/\tau)\dx\dt\,\\
    \label{6guys}&=:I_1(\ve)-QI_2(\ve)-I_3-I_4-I_5-I_6\,.
\end{align}
 To get the ultimate goal of Lemma~\ref{lem-7.2} we will prove that
\begin{align}\nonumber
\int_{B_1\cap\{\ve\leq|\tebv|\leq\tau\}}\frac{|\teDbv|^2+G(|\teDbv|)}{|\tebv|^\gamma}\phi\dx& \lesssim I_1(\ve)-QI_2(\ve)\lesssim I_3+I_4+ I_5+I_6+|\btemu|(B_1)\,\\
&\lesssim \int_{B_1}|\teDbv|^\vk \dx+|\btemu|(B_1)\,.\label{aim-of-Lem-4.7}
\end{align}

Our aim is to estimate the terms on the right-hand side of~\eqref{6guys}. Note they are all finite and nonnegative, so we can apply Fubini's theorem. We will exploit~\eqref{Phi} several times. We can estimate\begin{align*}
    I_1(\ve)&=\frac{1}{\gamma}\int_{B_1}\left[\frac{1}{\max\{|\tebv|\,,\ve\}^\gamma}-\frac{1}{K^\gamma\tau^\gamma}\right]\left[(B(x):\teDbv):\teDbv\right]\Phi(|\tebv|/\tau)\phi\dx\\
    &\geq \frac{1}{\gamma}\int_{B_1\cap\{|\tebv|\geq\ve\}}\left[\frac{1}{|\tebv|^\gamma}-\frac{1}{K^\gamma\tau^\gamma}\right]\left[(B(x):\teDbv):\teDbv\right]\Phi(|\tebv|/\tau)\phi\dx\,.
\end{align*}
Since $\Phi(|\tebv|/\tau)=0$ when $|\tebv|\geq 6{\tau}$, we infer that
\[\frac{\Phi(|\tebv|/\tau)}{K^\gamma\tau^\gamma}\leq \frac{6^\gamma}{K^\gamma}\frac{\Phi(|\tebv|/\tau)}{|\tebv|^\gamma}\,.\]
We fix $K:=6\cdot 2^\frac{1}{\gamma}$ and obtain
\begin{align}
    \nonumber I_1(\ve)&\geq \frac{1}{\gamma}\int_{B_1 {\cap \lbrace |\tebv| \geq \varepsilon \rbrace}}\left[1-\frac{6^\gamma}{K^\gamma}\right]\frac{(B(x):\teDbv):\teDbv}{{|\tebv|^\gamma}}\Phi(|\tebv|/\tau)\phi\dx\\
    \label{I1}&\geq \frac{1}{2\gamma}\int_{B_1\cap\{|\tebv|\geq\ve\}} \frac{(B(x):\teDbv):\teDbv}{{|\tebv|^\gamma}}\Phi(|\tebv|/\tau)\phi\dx\,.
\end{align}
As
\[I_2(\ve)=\frac{1}{1-\gamma}\int_{B_1\cap\{|\tebv|\geq\ve\}}\left[\max\{|\tebv|\,,\ve\}^{1-\gamma}-\ve^{1-\gamma}\right]\frac{(B(x):\teDbv):\teDbv}{|\tebv|}\Phi(|\tebv|/\tau)\phi\dx\,,\]
 we have that
\begin{align}
\label{I2}I_2(\ve)&\leq \frac{1}{1-\gamma}\int_{B_1\cap\{|\tebv|\geq\ve\}} \frac{(B(x):\teDbv):\teDbv}{|\tebv|^\gamma}\Phi(|\tebv|/\tau)\phi\dx\,,\\
\label{I3}I_3&\leq \frac{c}{\gamma\tau^{1+\gamma}}\int_{B_1\cap\{3\tau<|\tebv|<6\tau\}}\frac{g(|\teDbv|+1)}{|\teDbv|+1}|\teDbv|^2|\tebv|\phi\dx\,,\\
\label{I4}I_4&\leq \frac{c}{(1-\gamma)\tau}\int_{B_1\cap\{3\tau<|\tebv|<6\tau\}}\frac{g(|\teDbv|+1)}{|\teDbv|+1}|\teDbv|^2|\tebv|^{1-\gamma}\phi\dx\,,\\
I_5&\leq \frac{c}{\gamma} \int_{B_1}\frac{g(|\teDbv|+1)}{|\teDbv|+1}|\teDbv|\,|\tebv|^{1-\gamma}\,\Phi(|\tebv|/\tau)\, |D\phi|\dx\,,\nonumber\\
I_6&\leq \frac{c}{1-\gamma}\int_{B_1}\frac{g(|\teDbv|+1)}{|\teDbv|+1}|\teDbv|\,|\tebv|^{1-\gamma}\,\Phi(|\tebv|/\tau)\,|D\phi|\dx\,.\nonumber
\end{align}
For $\gamma< \gamma_0 := \frac{1}{4 (q-1) + 1}$
it holds \[\frac{1}{2\gamma}-\frac{q-1}{1-\gamma}\geq \frac{1}{4\gamma} = c_{\gamma} >0\,.\]
Then, by~\eqref{6guys},~\eqref{I1} and~\eqref{I2}, we get
\begin{align}\nonumber
    c|\btemu|(B_1)+I_3+I_4+I_5+I_6&\geq I_1(\ve)-QI_2(\ve)\\
    \nonumber&\geq c_{\gamma}\int_{B_1\cap\{|\tebv|\geq\ve\}} \frac{(B(x):\teDbv):\teDbv}{|\tebv|^\gamma}\Phi(|\tebv|/\tau)\phi\dx\,\\
    &\geq  c_{\gamma} c\int_{B_1\cap\{\ve\leq|\tebv|\leq\tau\}}\frac{|\teDbv|^2+G(|\teDbv|)}{|\tebv|^\gamma}\phi\dx\,.\label{full}
\end{align}
On the other hand, by~\eqref{I3}, \eqref{I4}, and Lemma~\ref{lem-7.1} we obtain\begin{align}
\nonumber    I_3+I_4&\leq \frac{c}{\gamma\tau^{\gamma}}\int_{B_1\cap\{3\tau<|\tebv|<6\tau\}}\frac{g(|\teDbv|+1)}{|\teDbv|+1}|\teDbv|^2\phi\dx\\
\nonumber &\leq \frac{c}{\gamma\tau^{\gamma}}\int_{B_1\cap\{3\tau<|\tebv|<6\tau\}}\left(|\teDbv|^2+G(|\teDbv|)\right)\phi\dx\\
 \label{I3+I4}   &\leq \frac{c}{\gamma\tau^{\gamma}}\left(\int_{B_1}  |\teDbv|^\vk \dx+|\btemu|(B_1)\right)\,,
\end{align}
where the last estimate holds true due to Lemma~\ref{lem-7.1} with $6\tau$ instead of $\tau$ and the fact that $\phi$ vanishes
outside $B_{3/4}$. Moreover,\begin{align*}
    I_5+I_6&\leq \frac{c}{\gamma(1-\gamma)}\int_{B_1\cap\{|\tebv|<6\tau\} }\frac{g(|\teDbv|+1)}{|\teDbv|+1}|\teDbv|\,|\tebv|^{1-\gamma}\,|D\phi|\dx\,.
\end{align*}
We split the integral over $\{|\teDbv|>H\}$ and $\{|\teDbv|\leq H\}$ for large enough $H=H(\data)>3$ and estimate separately. Due to~\eqref{trivial-inclusions}, we observe that
\begin{align*}
    \wt{\text{I}}&:=\int_{B_1\cap\{|\tebv|<6\tau\}\cap \{|\teDbv|>H\}}\frac{g(|\teDbv|+1)}{|\teDbv|+1}|\teDbv|\,|\tebv|^{1-\gamma}\,|D\phi|\dx\,\\
    &\leq c\tau^{1-\gamma} \int_{B_1 \cap \{|\teDbv|>H\}}{g(|\teDbv|)}\dx \leq c \int_{B_1 \cap \{|\teDbu|>H-1\}}{g(|\teDbu|)}\dx\,.
\end{align*}
To estimate it further we will apply Lemma~\ref{lem-6.1} {with $s=0$}. Note that~\eqref{smallness} is satisfied due to~\eqref{normalized-bar-v-2} and~\eqref{Dbu-small}. In turn, it holds
\begin{align*}
    \wt {\text{I}}\leq\int_{B_{3/4} \cap \{|\teDbu| > H-1\} } g(|\teDbu|) \dx & \le c \ \int_{B_1 \cap \{|\teDbu| > 2\} } |\teDbu|^{\vk} \dx + c  |\btemu|(B_1)\\
 &\le c \ \int_{B_1 } |\teDbv|^{\vk} \dx + c  |\btemu|(B_1)\,
\end{align*} 
with a constant $c = c(\data) > 0$. Let us assume from now on that $\varkappa \in [1,\varkappa_0)$, where $\varkappa_0 = 2 - \gamma_0$. Then, it holds $(1-\gamma)\varkappa' > \varkappa$ and therefore, using H\"older, and Poincar\'e inequality we deduce 
\begin{align*}
    \wt{\text{II}}&:=\int_{B_1 \cap\{|\tebv|<6\tau\}\cap \{|\teDbv|\leq H\}}\frac{g(|\teDbv|+1)}{|\teDbv|+1}|\teDbv|\,|\tebv|^{1-\gamma}\,|D\phi|\dx\,\\
    &\leq c \frac{g(H+1)}{H+1}\int_{B_1 \cap\{|\tebv|<6\tau\}}|\teDbv|\,|\tebv|^{1-\gamma}\dx\\
    &\leq c \left(\int_{B_1}|\teDbv|^{\vk}\right)^\frac{1}{{\vk}}\left(\int_{B_1\cap\{|\tebv|<6\tau\}}|\tebv|^{(1-\gamma){\vk}'}\dx\right)^\frac{1}{{\vk}'}\\
    &\leq c \left(\int_{B_1}|\teDbv|^{\vk}\right)^\frac{1}{{\vk}}\left(\int_{B_1\cap\{|\tebv|<6\tau\}}|\tebv|^{{\vk}}\dx\right)^\frac{1}{{\vk}'}\\
    &\leq c \left(\int_{B_1}|\teDbv|^{\vk} \right)^{\frac{1}{{\vk}}+\frac{1}{{\vk}'}} = c \int_{B_1} |\teDbv|^{{\vk}}\dx\,,
    \end{align*}
    where $c = c(\data,\gamma,\tau,\varkappa) > 0$ is a constant.
Altogether we have
\begin{align}\label{I5+I6}
    I_5+I_6\leq\frac{c}{\gamma(1-\gamma)}\left[
    \wt{\text{I}}+
    \wt{\text{II}}\right] &\le \frac{c}{\gamma(1-\gamma)}\left[ \int_{B_1 } |\teDbv|^{{\vk}} \dx + |\btemu|(B_1)\right]\,.
\end{align}
Summing up, by the use of~\eqref{full}, \eqref{I3+I4}, and \eqref{I5+I6} in~\eqref{aim-of-Lem-4.7}, we get the claim.
\end{proof}

\begin{lem} \label{lem-7.3}
If $\tebu$, $\tebv$ and $\btemu$ are as in \eqref{bar-system}--\eqref{normalized-bar-v-2}, then there exist constants $t=t(n, p, q)\in(1, \min\lbrace n', q' \rbrace)$, $\delta=\delta(n, p, q)  \in (0, 1)$ and $C=C(\data) > 0$ such that
\begin{equation} \label{eq-g-der}
\int_{B_{1/4}} \left( g''(1+|\teDbu|)|\teDbv|^2\right)^t\dx\leq C\left(\int_{B_{1}} |\teDbv|\dx\right)^{t+\delta-1}\left(\int_{B_{1}} |\teDbv|\dx+|\btemu|(B_1)\right)\,.
\end{equation}
\end{lem}

\begin{proof}
We start with choosing parameters $\gamma, s, t, \delta, \varepsilon_{1}$ and $\varepsilon_{2}$. Let  
\begin{equation} \label{eq:gamma}
0<\gamma <  \min \left\lbrace n'/(q-1),\gamma_0\right\rbrace\,,
\end{equation}
where $\gamma_0$ is the constant given in Lemma~\ref{lem-7.2}, and let $s \in [0, \sa)$ be admissible in Theorem~\ref{theo:Ah-approx}. We take $t \in (1, \min\lbrace n', q' \rbrace)$ and $\delta \in (0,1)$ so that
\begin{align} \label{eq:t}
1&<t < \min \left\lbrace  {q'  n'}/({\gamma+ n'}), 1+s \right\rbrace\,,\\
\label{eq:t-delta}
0&<\delta \leq 1+ \frac{\gamma }{(\gamma+ n')(n-1)}-t\,.
\end{align}
This is possible because the right-hand side of \eqref{eq:t} is larger than 1 by \eqref{eq:gamma}. We then define
\begin{equation*}
\varepsilon_{1}: = \frac{2\gamma}{\gamma+ n'} \quad\text{and}\quad \varepsilon_{2} := \frac{q\gamma}{\gamma+ n'}\,.
\end{equation*}
Then, we have
\begin{equation*}
\frac{\gamma(2-\varepsilon_{1})}{\varepsilon_{1}} = \frac{\gamma(q-\varepsilon_{2})}{\varepsilon_{2}} =  n'
\end{equation*}
and by \eqref{eq:t} we have
\begin{equation} \label{eq:eps2}
\varepsilon_{2} = q \left( 1-\frac{ n'}{\gamma+ n'} \right) \leq q\left( 1-\frac{t}{q'} \right) = q-(q-1)t\,.
\end{equation}
 Since the function
\begin{equation*}
f(t, s, \delta) := \frac{t}{1+s} +  n'\left( 1- \frac{t}{1+s} \right) - t -\delta
\end{equation*}
is decreasing in $t$ and $\delta$, and satisfies
\begin{equation*}
\lim_{t \searrow 1, \delta \searrow 0} f(t, s, \delta) = ( n'-1) \left( 1-\frac{1}{1+s}\right) > 0\,,
\end{equation*}
we may have $f \geq 0$, or equivalently,
\begin{equation} \label{eq:s02}
t+\delta \leq \frac{t}{1+s} +  n' \left( 1-\frac{t}{1+s} \right)\,,
\end{equation}
by assuming that $t$ and $\delta$ have been chosen sufficiently close to 1 and 0, respectively. We may assume that
\begin{equation}\label{eq:st}
s>t-1
\end{equation}
since we chose $t$ sufficiently close to 1.

Let us now set
\begin{equation*}
\EDbv:= \int_{B_{1}} |\teDbv| \dx \leq |B_{1}|\,,
\end{equation*}
and prove \eqref{eq-g-der} with $t$ and $\delta$ chosen as above. We split the integral on the left-hand side of \eqref{eq-g-der} as
\begin{equation*}
\begin{split}
\int_{B_{1/4}} \left( g''(1+|\teDbu|) |\teDbv|^{2} \right)^{t} \dx
&= \int_{B_{1/4} \cap \lbrace |\tebv| < 3 \rbrace \cap \lbrace |\teDbv| \leq H \rbrace} \left( g''(1+|\teDbu|) |\teDbv|^{2} \right)^{t} \dx \\
&\quad+ \int_{B_{1/4} \cap \lbrace |\tebv| < 3 \rbrace \cap \lbrace |\teDbv| > H \rbrace} \left( g''(1+|\teDbu|) |\teDbv|^{2} \right)^{t} \dx \\
&\quad+ \int_{B_{1/4} \cap \lbrace |\tebv| \geq 3 \rbrace \cap \lbrace |\teDbv| \leq H \rbrace} \left( g''(1+|\teDbu|) |\teDbv|^{2} \right)^{t} \dx \\
&\quad+ \int_{B_{1/4} \cap \lbrace |\tebv| \geq 3 \rbrace \cap \lbrace |\teDbv| > H \rbrace} \left( g''(1+|\teDbu|) |\teDbv|^{2} \right)^{t} \dx \\
&= J_{1} + J_{2} + J_{3} + J_{4}\,,
\end{split}
\end{equation*}
where $H \geq 3$ is a constant to be fixed later. It is enough to show that there exists a constant $C > 0$ such that
\begin{equation*}
J_{1} + J_{2} + J_{3} + J_{4} \leq C (\EDbv)^{t+\delta} (1+(\EDbv)^{-1} | \btemu(B_{1})|)\,.
\end{equation*} 
Let us first estimate $J_{1}$ and $J_{2}$. If $|\teDbv| \leq H$, then we have
\begin{equation*}
1 \leq 1+|\teDbu| \leq 1+|\teDbv|+|\teDbv-\teDbu| \leq 1+H+|D\teell| = 2+H\,.
\end{equation*}
In this case,  using \eqref{a:p-2}, we obtain
\begin{equation} \label{eq:J1-g}
g''(1+|\teDbu|) \leq \frac{q-2}{p-2} (1+|\teDbu|)^{q-3} g''(1) \leq C
\end{equation}
for some $C = C(p, q, g''(1), H) > 0$. This leads us to
\begin{equation*}
J_{1} \leq C \int_{B_{1/4} \cap \lbrace |\tebv| < 3 \rbrace \cap \lbrace |\teDbv| \leq H \rbrace} |\teDbv|^{2t} \dx\,.
\end{equation*}
Moreover, since $|\teDbv|^{2t} \leq H^{2t-2+\varepsilon_{1}} |\teDbv|^{2-\varepsilon_{1}}$, we have
\begin{equation} \label{eq:J1}
J_{1} \leq C \int_{B_{1/4} \cap \lbrace |\tebv| < 3 \rbrace} |\teDbv|^{2-\varepsilon_{1}} \dx\,.
\end{equation}
For $J_{2}$, we observe the following: if $|\teDbv| > H$, then we have
\begin{equation*}
|\teDbv| \leq |\teDbv-\teDbu|+|\teDbu| \leq 1+|\teDbu| \leq 2+|\teDbv| < 2|\teDbv|\,,
\end{equation*}
and hence, by~\eqref{a:p-2}, it holds
\begin{align*}
g''(1+|\teDbu|)= g''\left( \frac{1+|\teDbu|}{|\teDbv|} |\teDbv| \right) \leq \frac{q-2}{p-2} \left( \frac{1+|\teDbu|}{|\teDbv|} \right)^{q-3} g''(|\teDbv|) \leq \frac{q-2}{p-2} \max \lbrace 1, 2^{q-3} \rbrace g''(|\teDbv|)\,.
\end{align*}
Also by \eqref{a:p-2}, we obtain
\begin{equation*}
J_{2} \leq C \int_{B_{1/4} \cap \lbrace |\tebv| < 3 \rbrace \cap \lbrace |\teDbv| > H \rbrace} g^{t}(|\teDbv|) \dx\,.
\end{equation*}
Note that by \eqref{a:p} and \eqref{eq:eps2}, we know that for $a \geq 1$ it holds
\begin{equation*}
\frac{g^{t}(a)}{G^{1-\varepsilon_{2}/q}(a)} \leq q^{t} \frac{G^{t-1+\varepsilon_{2}/q}(a)}{a^{t}} \leq q^{t} G^{t-1+\varepsilon_{2}/q}(1) a^{q(t-1+\varepsilon_{2}/q)-t} \leq C \,.
\end{equation*}
 Therefore, we have
\begin{equation} \label{eq:J2}
J_{2} \leq C \int_{B_{1/4} \cap \lbrace |\tebv| < 3 \rbrace} G^{1-\varepsilon_{2}/q}(|\teDbv|) \dx\,.
\end{equation}

It follows from \eqref{eq:J1}, \eqref{eq:J2} and H\"older's inequality that
\begin{align*}
J_{1} &\leq C \left( \int_{B_{1/4} \cap \lbrace |\tebv| < 3 \rbrace} \frac{|\teDbv|^{2}}{|\tebv|^{\gamma}} \dx \right)^{1-\varepsilon_1/2} \left( \int_{B_{1}} |\tebv|^{\gamma(2-\varepsilon_1)/\varepsilon_1} \dx \right)^{\varepsilon_1/2}\,, \\
J_{2} &\leq C \left( \int_{B_{1/4} \cap \lbrace |\tebv| < 3 \rbrace} \frac{G(|\teDbv|)}{|\tebv|^{\gamma}} \dx \right)^{1-\varepsilon_2/q} \left( \int_{B_{1}} |\tebv|^{\gamma(q-\varepsilon_2)/\varepsilon_2} \dx \right)^{\varepsilon_2/q}.
\end{align*}
By using Lemma~\ref{lem-7.2} and the Sobolev--Poincar\'e inequality, we estimate
\begin{align*}
J_{1}+J_{2}
&\leq C (\EDbv+| \btemu|(B_{1}))^{1-\gamma/(\gamma+ n')} \left( \int_{B_{1}} |\tebv|^{ n'} \dx \right)^{\gamma/(\gamma+ n')} \\
&\leq C (\EDbv+| \btemu|(B_{1}))^{1-\gamma/(\gamma+ n')} (\EDbv)^{ n' \gamma/(\gamma+ n')} \\
&= C (1+(\EDbv)^{-1}| \btemu|(B_{1}))^{1-\gamma/(\gamma+ n')} (\EDbv)^{1+( n' -1)\gamma/(\gamma+ n')}\,.
\end{align*}
Using $\EDbv \leq |B_{1}|$ and \eqref{eq:t-delta}, we arrive at
\begin{equation*}
J_{1} + J_{2} \leq C (\EDbv)^{t+\delta} (1+(\EDbv)^{-1}| \btemu|(B_{1}))\,.
\end{equation*}

For $J_{3}$, we use \eqref{eq:J1-g}, $|\teDbv| \leq H$ and
\begin{equation} \label{eq:E}
|B_{1/4} \cap \lbrace |\tebv| \geq 3 \rbrace| \leq \frac{1}{3^{ n'}} \int_{B_{1/4}} |\tebv|^{ n'} \dx \leq C \left( \int_{B_{1}} |\teDbv| \dx \right)^{ n'} = C (\EDbv)^{ n'}
\end{equation}
to obtain
\begin{equation*}
J_{3} \leq C (\EDbv)^{ n'} \leq C (\EDbv)^{t+\delta} (1+(\EDbv)^{-1}| \btemu|(B_{1}))\,,
\end{equation*}
where we used $t+\delta \leq  n'$, which follows from \eqref{eq:t-delta}.

For $J_{4}$, we proceed as in the case of $J_{2}$: recalling \eqref{eq:st}, we use H\"older's inequality to obtain
\begin{align*}
J_{4}
&\leq C \int_{B_{1/4} \cap \lbrace |\tebv| \geq 3 \rbrace \cap \lbrace |\teDbv| > H \rbrace} g^{t}(|\teDbv|) \dx \\
&\leq C |B_{1/4} \cap \lbrace |\tebv| \geq 3 \rbrace \cap \lbrace |\teDbv| > H \rbrace|^{1-t/(1+s)} \left( \int_{B_{1/4} \cap \lbrace |\tebv| \geq 3 \rbrace \cap \lbrace |\teDbv|>H \rbrace} g^{1+s}(|\teDbv|) \dx \right)^{t/(1+s)}.
\end{align*}
Thus, it follows from \eqref{eq:E} that
\begin{equation*}
J_{4} \leq C (\EDbv)^{ n'(1-t/(1+s))} \left( \int_{B_{1/4} \cap \lbrace |\tebv| \geq 3 \rbrace \cap \lbrace |\teDbv|>H \rbrace} {g}^{1+s}(|\teDbv|) \dx \right)^{t/(1+s)}.
\end{equation*}
Let $c_{h}$ be the constant given in Lemma~\ref{lem-5.1} and set $H := 20^{n} 2c_{h}$, then
\begin{equation*}
H \geq 20^{n} c_{h} \left( \mean{B_{1}} |\teDbv| \dx + g^{-1}(| \btemu|(B_{1})) \right)\,.
\end{equation*}
By Lemma~\ref{lem-6.1} we obtain
\begin{equation*}
\int_{B_{1/4} \cap \lbrace |\teDbv| > H \rbrace} {g}^{1+s}(|\teDbv|) \dx \leq C (\EDbv+| \btemu|(B_{1}))\,,
\end{equation*}
and therefore
\begin{equation*}
J_{4} \leq C (\EDbv)^{ n'(1-t/(1+s))+t/(1+s)} (1+(\EDbv)^{-1} | \btemu|(B_{1}))^{t/(1+s)}.
\end{equation*}
By using \eqref{eq:s02} and recalling \eqref{eq:st}, we conclude
\begin{equation*}
J_{4} \leq C (\EDbv)^{t+\delta} (1+(\EDbv)^{-1}| \btemu|(B_{1}))\,.
\end{equation*}
\end{proof}

\begin{proof} [Proof of Proposition~\ref{prop-7.1}]
Let $t, \delta$ be given as in Lemma~\ref{lem-7.3}.    We look for a constant $\theta_{\mathrm{nd}}$ of the form $\theta_{\mathrm{nd}}:= (\varepsilon/H)^{t/\delta}$, that is, we assume
\begin{align} \label{eq:7.1-assumption1}
\mean{B_{r}} |D\teu - (D\teu)_{B_{r}}| \dx &\leq \left( \frac{\varepsilon}{H} \right)^{t/\delta} |(D\teu)_{B_{r}}|\,,\\
\label{eq:7.1-assumption2}
\frac{|\temu|(B_{r})}{r^{n-1}} &\leq \left( \frac{\varepsilon}{H} \right)^{t/\delta} \frac{g(|(D\teu)_{B_{r}}|)}{|(D\teu)_{B_{r}}|} \mean{B_{r}} |D\teu-(D\teu)_{B_{r}}| \dx\,.
\end{align}
We begin by rescaling
\begin{equation*}
\tebu(x) := \frac{\teu(x_{0}+rx)-(\teu)_{B_{r}}}{r |(D\teu)_{B_{r}}|} \qquad\text{and}\qquad \btemu(x) := \frac{r \temu(x_{0}+rx)}{g(|(D\teu)_{B_{r}}|)}\,.
\end{equation*}
Then, $\tebu$ solves the system
\begin{equation*}
-\tedv\left( \frac{\overline{g}(|D\tebu|)}{|D\tebu|} D\tebu \right) = \btemu \quad\text{in } B_{1}\,,
\end{equation*}
where
\begin{equation*}
\overline{g}(t) := \overline{g}_{|(D\teu)_{B_{r}}|}(t) \quad\text{and}\quad \overline{G}(t) := \overline{G}_{|(D\teu)_{B_{r}}|}(t)
\end{equation*}
are defined as in~\eqref{eq:gGH}. Moreover, we have
\begin{equation*}
|(D\tebu)_{B_{1}}| = 1\,, \quad (\tebu)_{B_{1}} = 0\,, \quad \text{and}\quad \overline{\cL}((D\tebu)_{B_{1}}) = \cL((D\teu)_{B_{r}})\,,
\end{equation*}
where $\overline{\cL}$ is defined by \eqref{eq:L} with $g$ replaced by $\overline{g}$. The assumptions \eqref{eq:7.1-assumption1} and \eqref{eq:7.1-assumption2} read
\begin{align}
\wt E:= \mean{B_{1}} |D\tebu - (D\tebu)_{B_{1}}| \dx &\leq \left( \frac{\varepsilon}{H} \right)^{t/\delta}\nonumber,\\
|\btemu|(B_{1}) \leq \left( \frac{\varepsilon}{H} \right)^{t/\delta}\wt E &\leq 1\,.\label{meas-ex-1}
\end{align} 
We next define the maps $\teell: = (\ell^{\alpha})_{1\leq \alpha \leq m}$ as
\begin{equation*}
\ell^{\alpha}(x) := \langle (D\tebu^{\alpha})_{B_{1}}, x \rangle\,,
\end{equation*}
and $\tebv = {\tebu}-\teell$. Then $|D\teell| = |(D\tebu)_{B_{1}}| = 1$ and
\begin{equation*}
\wt E = \EDbv= \mean{B_{1}} |D\tev| \dx \leq \left( \frac{\varepsilon}{H} \right)^{t/\delta} \leq 1\,.
\end{equation*}
Let us now define $\overline{\teh} \in W^{1, 2}(B_{1/4})$ as the solution to
\begin{equation*}
\begin{cases}
-\tedv(\overline{A}: D\overline{\teh}) = 0 &\text{in } B_{1/4}\,, \\
\overline{\teh} = \tebu &\text{on } \partial B_{1/4}\,,
\end{cases}
\end{equation*}
where \[\overline{A} := \overline{\cL}((D\tebu)_{B_{1}}) = \cL((D\teu)_{B_{r}})\,.\] Since $D\teell = (D\tebu)_{B_{1}}$ and $|D\teell|=1$, we have
\begin{equation}\label{eq-linear-sys}
\begin{split}
\tedv(\overline{A}: (D\overline{\teh}-D\tebu))
&= -\tedv(\overline{A}: D\tebu) +\tedv\left( \frac{\overline{g}(|D\tebu|)}{|D\tebu|} D\tebu \right) + \btemu \\
&=\tedv\left( \frac{\overline{g}(|D\tebu|)}{|D\tebu|} D\tebu - \frac{\overline{g}(|D\teell|)}{|D\teell|} D\teell - \frac{\overline{g}(|D\teell|)}
{|D\teell|} \overline{A}: (D\tebu-D\teell) \right) + \btemu \\
&=:\tedv \, \teW + \btemu\,.
\end{split}
\end{equation}
Since by~\eqref{partial-2-A-bound} we have that $|\partial^{2}\opA(\texi)| \leq cg''(|\texi|)$, using  Lemma~\ref{lem-g''}, and $|D\teell|=1$, we obtain
\begin{align*} 
|\teW| &=\left| \int_{0}^{1} (\partial \opA(D\teell + \tau(\teDbu-D\teell)) - \partial \opA(D\teell)) \dtau : (\teDbu - D\teell)\right|\\
&\leq \int_0^1|\partial \opA (D\teell + \tau(D\tebu-D\teell)) - \partial \opA(D\teell)|\dtau\,|D\tebu-D\teell|\\
&\leq \int_0^1\int_{0}^{1} |\partial^{2} \opA(D\teell + \tau s (D\tebu - D\teell))| \ds\dtau\,|D\tebu-D\teell|^2 \\
&\leq c \int_0^1 \int_{0}^{1} g''(|D\teell + \tau s (D\tebu - D\teell)|) \ds \dtau\, |D\tebu-D\teell|^2 \\
&\leq c g''(|D\teell|+|D\tebu|) |D\tebu-D\teell|^2 \,.
\end{align*} 
Hence we estimate
\begin{equation*}
|\teW| \leq c g''(1+|D\tebu|) |D\tev|^{2}.
\end{equation*}
By the Calder\'on--Zygmund theory\footnote{The estimate \eqref{eq-CZ} when $\teW$ does not vanish and $\temu=0$ follows from the result in \cite{DM93,Ste70} and a standard duality argument \cite{BBDL22}. On the other hand, its counterpart for scalar problems when $\teW=0$ and $\temu$ is a measure can be found in \cite{BG}. The proof in \cite{BG} easily generalizes to the vectorial case. The full estimate \eqref{eq-CZ} is obtained by combining these two results using the linearity of  \eqref{eq-linear-sys}. See also \cite{DFFZ13}.}, there exists a constant $c>0$ such that
\begin{equation}\label{eq-CZ}
\left( \mean{B_{1/4}} |D\tebu - D\overline{\teh}|^{t} \dx \right)^{1/t} \leq c \left( \mean{B_{1/4}} |\teW|^{t} \dx \right)^{1/t} + c |\btemu|(B_{1})\,.
\end{equation}
By Lemma~\ref{lem-7.3}, we have
\begin{equation*}
\begin{split}
\mean{B_{1/4}} |\teW|^{t} \dx
&\leq c \left( \int_{B_{1}} |D\tev| \dx \right)^{t+\delta-1} \left( \int_{B_{1}} |D\tev| \dx + |\btemu(B_{1})| \right) \leq c \wt E^{t+\delta} \leq c \left( \frac{\varepsilon}{H} \right)^{t} \wt E^{t}\,,
\end{split}
\end{equation*}
which in the view of the previous display and~\eqref{meas-ex-1} yields
\begin{equation*}
\left( \mean{B_{1/4}} |D\tebu - D\overline{\teh}|^{t} \dx \right)^{1/t} \leq c \left( \frac{\varepsilon}{H} \right) \wt E + c \left( \frac{\varepsilon}{H} \right)^{t/\delta} \wt E \leq c \left( \frac{\varepsilon}{H} \right) \wt E\,.
\end{equation*}
Scaling back to original solutions, setting also
\begin{equation*}
\teh(x) = r \,|(D\teu)_{B_{r}}| \,\overline{\teh}\left( \frac{x-x_{0}}{r} \right), \quad x \in B_{r/4}(x_{0})\,,
\end{equation*}
we arrive at
\begin{equation*}
\left( \mean{B_{r/4}} |D\teu - D\teh|^{t} \dx \right)^{1/t} \leq c \frac{\varepsilon}{H} \mean{B_{r}} |D\teu - (D\teu)_{B_{r}}| \dx\,.
\end{equation*}
Hence, choosing $H:=c$ finishes the proof.
\end{proof}

\section{Excess decay}\label{sec:SOLA}

The goal of this section is to establish an excess decay estimate, from which the potential estimate Theorem \ref{theo:pointwise} follows by standard arguments. The main results of this section are Lemmas \ref{lem-8.4} and \ref{lem-8.5}.

\subsection{Auxiliary results}

We make use of the classical estimates for gradients of $A$-harmonic maps.
\begin{prop}[Chapter 10 in \cite{giusti}]\label{prop:osc-Ah} Suppose that $\teh$  is  $A$-harmonic in $B_r \subset \Omega$, where $A$ is uniformly elliptic, i.e., for some $c_A\geq 1$ it holds
\[c_A^{-1}|\texi|^2\leq (A:\texi):\texi\leq c_A|\texi|^2\qquad\forall \texi\in\rnm\,.\] Then there exists a constant $c_{\mathrm{hol},A} > 0$ depending only on $n,m,c_A$ such that
\begin{align*}
    \osc_{B_{\delta r}} D\teh \le c_{\mathrm{hol},A} \delta \mean{B_r} |D\teh - (D\teh)_{Br}| \dx \quad \forall \delta \in (0,1/2]\,.
 \end{align*}
\end{prop}
We would like to stress that the regularity of $\opA$-harmonic maps with $\opA$ as in~\eqref{opA:def} is an object of intensive study \cite{DiEt,DSV1,DSV2,DiLeStVe}. Inspired by these contributions we infer the following couterpart of the above proposition.
\begin{lem}\label{lem:grad-osc} 
Suppose Assumption~\ref{ass:G} is satisfied and $\tev\in W^{1,G}(\Omega,\Rm)$ is $\opA$-harmonic in $B_r \subset \Omega$ for $\opA$ given by~\eqref{opA:def}. Then there exist constants $c_{\mathrm{hol},G} > 0$, $\alpha_{\mathrm{hol}} \in (0,1)$, and $\sigma_0 \in (0,1/2)$ depending only on $n,m,p,q$ such that
\begin{align*}
    \osc_{B_{\delta r}} D\tev \le c_{\mathrm{hol},G} \delta^{\alpha_{\mathrm{hol}}} \mean{B_r} |\teDv - (\teDv)_{Br}| \dx \quad \forall \delta \in (0,\sigma_0]\,.
 \end{align*}
\end{lem}

\begin{proof}
    By Theorem 6.1 in \cite{DSV1} (application of Theorem 6.4 in \cite{DSV1} and using Campanato's theorem), we obtain
    \begin{align*}
        \osc_{B_{\delta r}} \calV(\teDv) \le c \delta^{\alpha_{\mathrm{hol}}} \mean{B_{r/2}} \calV(\teDv)\dx\,,
    \end{align*}
    where $\calV({\xi})= \left(\frac{g(|{\xi}|)}{|{\xi}|}\right)^{1/2}\xi$. Since $p > 2$, by \cite[Theorem 1 and Corollary 1.2]{BoCh} it holds
    \begin{align*}
        |\calV(\texi) - \calV(\teeta)| \ge c G^{1/2}(|\texi - \teeta|), \quad |\calV(\texi)| = G^{1/2}(|\texi|)\,.
    \end{align*}
    Thus, we obtain
    \begin{align*}
        G^{1/2}\left(\osc_{B_{\delta r}} \teDv\right) &\le c \delta^{\alpha} \mean{B_{r/2}} G^{1/2}(|\teDv|)\dx\le c \delta^{\alpha} G^{1/2}\left(\sup_{B_{r/2}} |\teDv|\right)\,.
    \end{align*}
    As a consequence, applying $(G^{1/2})^{-1}$ on both sides, and applying Lemma \ref{lem:Lip} we end up with
    \begin{align*}
        \osc_{B_{\delta r}} \teDv \le c \delta^{\alpha_{\mathrm{hol}}} \sup_{B_{r/2}} |\teDv| \le c \delta^{\alpha_{\mathrm{hol}}} \mean{B_{r}} |\teDv|\dx\,.
    \end{align*}
    With this estimate at hand, by following the arguments in the proof of \cite[Theorem 3.2]{KuMi2018}, we obtain the following two results: 
    \begin{align*}
    c^{-1} |(\teDv)_{B_r}| &\le |\teDv(x)| \le c |(\teDv)_{B_r}| \qquad \forall x \in B_{2 \sigma_0 r}\,,\\
        |\teDv(x) - \teDv(y)| &\le c |(\teDv)_{B_r}||x-y|^{\alpha_{\mathrm{hol}}} \qquad \forall x,y \in B_{r/2}\,,
    \end{align*}
    where $c > 0$ depends only on $n,m,p,q$ and $\sigma_0$ is a constant depending only on $\alpha_{\mathrm{hol}}$.\\
    Let us observe that $\tew_i := \partial_{x_i}\tev$ satisfies the system
    \begin{align*}
        -\tedv(\cB(x)\cL(\teDw_i)) = 0 \qquad\text{with}\qquad \cB(x) = {g\left( \frac{|\teDv|}{(\teDv)_{B_r}} \right)}{\frac{(\teDv)_{B_r}} {|\teDv|}} \mathcal{L}(\teDv(x))\,,
    \end{align*}
    where $\mathcal{L}$ is defined as in \eqref{eq:L} and it satisfies \eqref{L-basic-prop} and \eqref{eq:L-upper}. Using the previous two estimates, we obtain that 
    \begin{align*}
        c^{-1}|\texi|^2 \le (\cB(x):\texi):\texi \le c |\texi|^2 \qquad \forall \texi \in \rnm ~~ x \in B_{2\sigma_0 r}\,,
    \end{align*}
    making it a linear elliptic system. Moreover, $\cB$ is H\"older continuous by the previous two estimates and since $g$ is $C^2$. Thus, by a standard perturbation argument, $w_i$ is locally H\"older continuous in $B_{\sigma_0 r}$ for any exponent $\alpha \in (0,1)$, and we have
    \begin{align*}
        \osc_{B_{\delta r}} \teDw_i \le c \delta^{\alpha_{\mathrm{hol}}} \mean{B_{2\sigma_0 r}} |\teDw_i - (\teDw_i)_{B_{\sigma_0 r}}|\dx \le 2 c \sigma_0^{-n} \delta^{\alpha_{\mathrm{hol}}} \mean{B_r} |\teDw_i - (\teDv)_{B_r}|\dx
    \end{align*}
    for any $\delta \in (0, \sigma_0)$. This concludes the proof.
\end{proof}

\subsection{Fixing parameters}\label{ssec:parameters}
First of all, we define 
\begin{align*}
    c_{\mathrm{hol}} = \max \{(c_{\mathrm{hol},A}),(c_{\mathrm{hol},G})\}\,,
\end{align*}
where $c_{\mathrm{hol},A}$ denotes the constant from  Proposition~\ref{prop:osc-Ah} and $c_{\mathrm{hol},G}$ is the constant from  Lemma~\ref{lem:grad-osc}.
Moreover, we fix
\begin{align}\label{sigma}
    \sigma &:=\min\left\{ \frac{1}{32}, \frac{\sigma_0}{20}, \left( \frac{1}{8^{n+20} c_{\mathrm{hol}}} \right)^{1/\alpha_{\mathrm{hol}}} \right\}\,,
\end{align}

Having $B_r(x_0)\subset\Omega$ for every $j=0,1,\dots$ we set\begin{equation*}
    r_j:=\sigma^{j+1}r\qquad\text{and}\qquad B^j:=\overline{B_{r_j}(x_0)}\,,
\end{equation*}
and $r_{-1}:=r$. In turn we obtain a sequence of closed balls shrinking to $x_0$
\[\cdots B_{r_{j+1}}(x_0)=B^{j+1}\subset B^j=B_{r_{j}}(x_0)\subset \cdots \subset B^0=B_{\sigma r}(x_0)\,.\]

Next, we define
\begin{align}\label{ve-and-theta}
    \varepsilon := \frac{\sigma^n}{2^{10}}\,,\qquad \theta := \min\left\{\theta_{\rm nd}\,, \sigma^{4n}  \right\},
\end{align}
where $\theta_{\rm nd}=\theta_{\rm nd}(\data,\ve)$ is the constant from Proposition~\ref{prop-7.1}.
Finally, set
\begin{align}\label{H1-and-H2}
    H_1 := \frac{2^8}{\theta}\,, \qquad H_2 := \frac{C(\data) c_{\rm d} H_1}{\sigma^{8n}\theta}\,,
\end{align}
where $C(\data) > 0$ is a constant that needs to be chosen large enough so that the estimates \eqref{eq:C-large1} and \eqref{eq:C-large2} hold true. Finally, recall the constant $c_{\rm d} > 0$ from Proposition~\ref{prop-5.1}.

\subsection{Gradient excess control}
We define
\begin{equation*}
    A_j:=|(\teDu)_{B^j}|\,,\quad 
    E_j:=\mean{B^j}|\teDu-(\teDu)_{B^j}|\dx\,, \quad\text{and}\quad C_j:=A_j+HE_j\,.
\end{equation*}

\begin{lem}
    \label{lem8.1} Suppose Assumption~\ref{ass:G} is satisfied. Let $\teu\in W^{1,G}(\Omega,\Rm)$ be a weak solution to~\eqref{eq:mu} in $\Omega$ with $\temu\in C^\infty(\Omega,\Rm)$,  $\theta$ be as in~\eqref{ve-and-theta}, $\sigma$ be as in~\eqref{sigma}, and $c_{\rm d}$ as in Proposition~\ref{prop-5.1}. If for some integer $j$ it holds\begin{equation}
        \label{Ej-large} E_j\geq \theta A_j\,,
    \end{equation}
    then\begin{equation*}
        E_{j+1}\leq \frac{E_j}{2^7}+\frac{2c_{\rm d}}{\sigma^n}g^{-1}\left(\frac{|\temu|(B^j)}{r_j^{n-1}}\right)\,.
    \end{equation*}
\end{lem}

\begin{proof}
By \eqref{Ej-large}, we can apply Proposition~\ref{prop-5.1}. In particular, since $g^{1+s}$ is convex, there exists an $\opA$-harmonic map $\tev_j$ in $B_{r_j/2}$ such that
\begin{align}
\label{eq:lem8.1-help1}
    \mean{B^j{/2}}|\teDu-\teDv_j|\dx \leq \varepsilon E_j+c_{\rm d}g^{-1}\left(\frac{|\temu|(B^j)}{r_j^{n-1}}\right)\,.
\end{align}
Moreover, by Lemma~\ref{lem:grad-osc} and using \eqref{eq:lem8.1-help1} we have
\begin{align}
\label{eq:lem8.1-help2}
\begin{split}
    \osc_{B^{{j+1}}} \teDv_j &\le 2 c_{\mathrm{hol}} \sigma^{\alpha_{\mathrm{hol}}} \mean{B^j{/2}} |\teDv_j - (\teDv_j)_{B^{j}}| \dx \\
    &\le 4 c_{\mathrm{hol}} \sigma^{\alpha_{\mathrm{hol}}} \mean{B^j{/2}} |{\teDv_j - (\teDu)_{B_{r_j}}}| \dx\\
    &{\le 4 c_{\mathrm{hol}} \sigma^{\alpha_{\mathrm{hol}}} \left( 2^n E_j + \mean{B^j{/2}} |\teDu - \teDv_j| \dx \right)} \\
    &\le 4 c_{\mathrm{hol}} \sigma^{\alpha_{\mathrm{hol}}} \left({2^{n+1}} E_j + {c_{\rm d}}\, g^{-1}\left(\frac{|\temu|({B^j})}{r_j^{n-1}}\right) \right).
    \end{split}
\end{align}
Having at hand \eqref{eq:lem8.1-help1} and \eqref{eq:lem8.1-help2}, we obtain
\begin{align*}
    E_{j+1} & \le \sigma^{-n} \mean{B^{j}/2} |\teDu - \teDv_j| \dx + 2 \osc_{{B^{j+1}}} \teDv_j\\
    &\le (2^{{n+3}} c_{\mathrm{hol}} \sigma^{\alpha_{\mathrm{hol}}} + \sigma^{-n} \varepsilon) E_j + {\left( 8 c_{\mathrm{hol}} c_{\rm d} \sigma^{\alpha_{\mathrm{hol}}} + \sigma^{-n} c_{\rm d} \right)} g^{-1}\left(\frac{|\temu|({B^j})}{r_j^{n-1}}\right) \\
    &\le 2^{-7} E_j + \frac{2 c_{\rm d}}{\sigma^{n}} g^{-1}\left(\frac{|\temu|({B^j})}{r_j^{n-1}}\right)
\end{align*}
by following the computation in \cite[Lemma 8.1]{KuMi2018} and using that 
\begin{align*}
    \sigma \le \left( \frac{1}{8^{n+20}c_{\mathrm{hol}}}\right)^{\frac{1}{\alpha_{\mathrm{hol}}}} \qquad\text{and}\qquad \varepsilon \le \frac{\sigma^n}{2^{10}}\,.
\end{align*}
\end{proof}

\begin{lem}
    \label{lem8.2} Suppose Assumption~\ref{ass:G} is satisfied. Let $\teu\in W^{1,G}(\Omega,\Rm)$ be a weak solution to~\eqref{eq:mu} in $\Omega$ with $\temu\in C^\infty(\Omega,\Rm)$, and $\theta$ be as in~\eqref{ve-and-theta}. If \begin{equation}
        \label{Ej-small,meas-nice} E_j\leq \theta A_j\qquad\text{and}\qquad \frac{|\temu|(B^j)}{r_j^{n-1}}\leq \theta \frac{g(A_j)}{A_j}E_j\,,
    \end{equation}
    then\begin{equation*}
        E_{j+1}\leq \frac{E_j}{4}\,.
    \end{equation*}
\end{lem}
\begin{proof}
    This fact follows the lines \cite[Lemma 8.2]{KuMi2018} upon slight modifications we shall expose. We start with an observation that without loss of generality we can assume that $A_j>0$. Indeed, when $A_j=0$, by~\eqref{Ej-small,meas-nice} we get that $E_j=0$, that is $ \teDbu$ is constant in $B^j$. Then also $E_{j+1}=0$ and the claim follows. We concentrate on the remaining case when $A_j>0$. We can apply Proposition~\ref{prop-7.1} in $B_r=B_{r_j}(x_0)$ since  $\theta{\leq}\theta_{\rm nd}$, and~\eqref{degeneracy-condition} and~\eqref{nice-measure} are satisfied due to~\eqref{Ej-small,meas-nice}, and $A=\cL((\teDu)_{B^j})$ is elliptic due to~\eqref{L-basic-prop}. We obtain an $A$-harmonic map $\teh_j\in W^{1,2}(B_{r_j/4})$ such that\[\mean{\frac 14 B^j}|\teDu-D\teh_j|\dx\leq \varepsilon E_j\,.\] By Proposition~\ref{prop:osc-Ah} we know that
    \[\osc_{B^{j+1}}D\teh_j\leq 4c_{\mathrm{hol}}\sigma^{\alpha_{\mathrm{hol}}}\mean{\frac 14 B^j}|D\teh_j-(D\teh_j)_{\frac 14 B^j}|\dx\leq 8c_{\mathrm{hol}}\sigma^{\alpha_{\mathrm{hol}}}\mean{\frac 14 B^j}|D\teh_j-(\teDu)_{\frac 14 B^j}|\dx\,.\]
    Now, we combine the two above displays after noticing that\begin{align*}
        E_{j+1}&\leq 2\mean{B^{j+1}}|\teDu-(D\teh_j)_{B^{j+1}}|\dx\leq  2\mean{B^{j+1}}|\teDu-D\teh_j|\dx+ 2\mean{B^{j+1}}|D\teh_j-(D\teh_j)_{B^{j+1}}|\dx,
    \end{align*}
    and proceed as in the proof of \cite[Lemma 8.2]{KuMi2018}. We get $E_{j+1}\leq C E_j$ with $C=(\sigma^{-n}\varepsilon+16c_{\mathrm{hol}}\sigma^{\alpha_{\mathrm{hol}}}(4^n+\ve))<1/4$ because of the choice of $\sigma$.
\end{proof}

\begin{lem}
    \label{lem8.3}Suppose Assumption~\ref{ass:G} is satisfied. Let $\teu\in W^{1,G}(\Omega,\Rm)$ be a weak solution to~\eqref{eq:mu} with $\temu\in C^\infty(\Omega,\Rm)$, and let $\theta$ be as in~\eqref{ve-and-theta}. If \begin{equation}
        \label{Ej-small,meas-bad} E_j\leq \theta A_j\qquad\text{and}\qquad \frac{|\temu|(B^j)}{r_j^{n-1}}\geq \theta \frac{g(A_j)}{A_j}E_j\,,
    \end{equation}
    then\begin{equation*}
        E_{j}\leq \frac{1}{\theta}g^{-1}\left(\frac{|\temu|(B^j)}{r_j^{n-1}}\right)\,.
    \end{equation*}
\end{lem}
\begin{proof}
    Note that using that $G$ is super-quadratic ($p> 2$), $\theta\leq 1$, and~\eqref{Ej-small,meas-bad} we have
    \[g(\theta E_j)=\frac{g(\theta E_j)}{\theta E_j}\theta E_j\leq  \frac{g(\theta^2 A_j)}{\theta A_j}E_j\leq \theta^{2(p-2)}\left(\theta \frac{ g(A_j)}{A_j}\right)E_j\leq \frac{|\temu|(B^j)}{r_j^{n-1}}\,.\]
\end{proof}

\subsection{Key lemmas for SOLA}

\begin{lem}[Lemma 8.4 in \cite{KuMi2018}]\label{lem-8.4}    
Suppose Assumption~\ref{ass:G} is satisfied. Let $\teu\in W^{1,G}(\Omega,\Rm)$ be a~SOLA to~\eqref{eq:mu} in $\Omega$ with a bounded Borel measure datum $\temu$, and $\sigma$ be as in \eqref{sigma}.  Let $\lambda > 0$ be such that for some $k \ge k_0 \ge 0$ it holds
    \begin{align}\nonumber
        C_j \le \lambda\,, \qquad C_{j+1} &\ge \lambda /16\,, \qquad \forall j \in \{k_0,\dots,k\}\,, \qquad C_{k_0} \le \lambda /4\,,\\
        \label{eq:lem8.4-ass2}
        g^{-1}\left(\sum_{j={k_0}}^k\frac{|\temu|(B^j)}{r_j^{n-1}}\right) &\le \frac{2\lambda}{H_2\sigma^n}\,.
    \end{align}

    Then
    \begin{align}
    \label{eq:lem8.4-goal}
        C_{k+1} < \lambda\,, \qquad \sum_{j = {k_0}}^{k+1} E_j \le \frac{\sigma^{4n}}{32}\lambda\,, \qquad \sum_{j = {k_0}}^{k+1} E_j \le 2 E_{k_0} + \frac{H_2 \sigma^n}{64 H_1} \frac{1}{g'(\lambda)} \sum_{j = {k_0}}^k \frac{|\temu|(B^j)}{r_j^{n-1}}\,.
    \end{align}
\end{lem}

\begin{proof}
    The proof follows along the lines of \cite[Lemma 8.4]{KuMi2018} and relies on Lemmas \ref{lem8.1}, \ref{lem8.2}, and \ref{lem8.3}. The only difference to \cite[Lemma 8.4]{KuMi2018} comes from the occurence of the growth function $g$ in the respective lemmas and in \eqref{eq:lem8.4-ass2}. Let us shortly sketch the main arguments of the proof:

\noindent{\bf    Step 0.} As in \cite{KuMi2018} it suffices to prove \eqref{eq:lem8.4-goal} for energy solutions $\teu_h \in W^{1,G}(\Omega,\Rm)$ and $\temu_h \in C^{\infty}(\Omega,\Rm)$ satisfying~\eqref{conv-of-meas} and then pass to the limit according to the definition of SOLA. Let us descibe how to do it. We  define
    \[A_j^h:=|(\teDu_h)_{B^j}|\,,\quad E_j^h:=\mean{B^j}|\teDu_h-(\teDu_h)_{B^j}|\dx\,,\quad\text{and}\quad C_j^h:=A_j^h+H_1E_j^h\,.\]
    Having $W^{1,1}$-convergence of $\{\teu_h\}$, we can justify that for every $j$ it holds
    \[A_j^h\to A_j\,,\quad E_j^h\to E_j\,,\quad C_j^h\to C_j\qquad\text{as }\ h\to\infty\,. \]
    If the assumptions are satisfied for $\teu_{\bar h}$ with $\bar h$ large enough, they are satisfied for $h>\bar h$. In turn, it is enough to prove the claim for the approximate solution with fixed $h$ and then let $h\to\infty$.

\noindent{\bf Step 1.} Degenerate case: Let us assume that
    \begin{align*}
        E_j \ge \theta A_j\,.
    \end{align*}
    By assumption, this implies $A_j \le \lambda / 2^6$. Moreover, by Lemma \ref{lem8.1}, \eqref{eq:lem8.4-ass2}, and the choice of $H_1$, we obtain $H_1 E_{j+1} \le \lambda / 2^6$. Besides, we have $|A_{j+1} - A_j| \le \sigma^{-n}E_j\le \lambda / 2^6$ by the same arguments as in \cite[Lemma 8.4]{KuMi2018} and the definition of $H_1$. This leads to a contradiction, since
   \begin{align*}
       \lambda / 16 \le C_{j+1} \le |A_{j+1} - A_j| + A_j + H_1 E_{j+1} < 3 \lambda / 2^6\,,
   \end{align*}
which implies that $E_j \le \theta A_j$ always holds true.

\noindent{\bf Step 2.} Nondegenerate case 1: 
   Clearly, if
   \begin{align*}
       E_j \le \theta A_j\qquad\text{and}\qquad \frac{|\temu|(B^j)}{r_j^{n-1}}\leq \theta \frac{g(A_j)}{A_j}E_j\,,
   \end{align*}
  then we obtain from Lemma~\ref{lem8.2} that
   \begin{align*}
       E_{j+1} \le E_j/4\,.
   \end{align*}

\noindent{\bf Step 3.} Nondegenerate case 2: Finally, we consider the case
   \begin{align*}
        E_j \le \theta A_j\qquad\text{and}\qquad \frac{|\temu|(B^j)}{r_j^{n-1}}\geq \theta \frac{g(A_j)}{A_j}E_j\,.
    \end{align*}
    By Lemma \ref{lem8.3} and \eqref{eq:lem8.4-ass2}, we obtain
    \begin{equation*}
        E_{j} \le \frac{1}{\theta}g^{-1}\left(\frac{|\temu|(B^j)}{r_j^{n-1}}\right) \le \frac{2\lambda}{\theta H_2\sigma^n}\,.
    \end{equation*}
    Using the above estimates, by elementary manipulations starting from
    \[\lambda/16\leq C_{j+1}\leq A_j+|A_{j+1}-A_j|+H_1 E_{j+1}\leq A_j+\lambda/32\,,\]
  see Step~3 of the proof of \cite[Lemma~8.4]{KuMi2018}, we establish that $A_j \ge \lambda / 2^5$. Thus, we get from \eqref{a:p-2}:
    \begin{align*}
        E_{j+1} \le \frac{2 E_j}{\sigma^n} \le \frac{2}{\theta \sigma^n (q-2)} \frac{1}{g'(A_j)} \frac{|\temu|(B^j)}{r_j^{n-1}} \le \frac{2^{1+5(p-2)}}{\theta \sigma^n (q-2)} \frac{1}{g'(\lambda)} \frac{|\temu|(B^j)}{r_j^{n-1}}\,.
    \end{align*}

\noindent{\bf Step 4.}  Conclusion:
    Altogether, we have
    \begin{align*}
        E_{j+1} \le E_j / 4 + \frac{2^{1+5(p-2)}}{\theta \sigma^n (q-2)} \frac{1}{g'(\lambda)}\frac{|\temu|({B^j})}{r_j^{n-1}}\,.
    \end{align*}
    By summation and using \eqref{eq:lem8.4-ass2}, we have
    \begin{align}
    \label{eq:C-large1}
    \begin{split}
        \sum_{j = {k_0}}^{k+1} E_j &\le \frac{4E_{k_0}}{3} + \frac{C(p,q) H_2 \sigma^n }{H_1} \frac{1}{g'(\lambda)}\sum_{j={k_0}}^k \frac{|\temu|({B^j})}{r_j^{n-1}}\\
        &\le \frac{\lambda}{3H_1} + \frac{C(p,q) H_2 \sigma^n }{H_1} \frac{1}{g'(\lambda)} g\left(\frac{2\lambda}{H_2 \sigma^n}\right)\\
        &\le \frac{\lambda}{3H_1} + \lambda \frac{C(p,q)(H_2 \sigma^n)^{2-p}}{H_1} \le \frac{5\lambda}{12 H_1}\,,
    \end{split}
    \end{align}
    where $C(p,q) > 0$ is a constant depending only on $p,q$. Moreover, we used \eqref{a:p-1}, $g(\lambda)/g'(\lambda)\approx \lambda$, and made use of the appropriate choice of $H_1$ and $H_2$, choosing the constant $C(\data) > 0$ in the definition of $H_2$ large enough depending on $C(p,q)$. We complete the proof by noticing
    \begin{align*}
        C_{k+1}&\leq A_{k_0}+\sum_{j={k_0}}^k|A_{j+1}-A_j|+H_1 E_{k+1}\leq C_{k_0}+\sum_{j={k_0}}^k\mean{B^{j+1}}|\teDu-(\teDu)_{B^j}|\dx+H_1 E_{k+1}\\
        &\leq \tfrac{\lambda}{4}+(\sigma^{-n}+H_1)\sum_{j={k_0}}^k E_j\leq \left(\tfrac{1}{2}+\tfrac{5}{12}\right)\lambda=\tfrac{11}{12}\lambda\,.
    \end{align*}
\end{proof}

\begin{lem}\label{lem-8.5}
Suppose Assumption~\ref{ass:G} is satisfied. Let $\teu\in W^{1,G}(\Omega,\Rm)$ be a SOLA to~\eqref{eq:mu} in $\Omega$ with a bounded Borel measure datum $\temu$. 
There exist positive constants $c_V=c_V ({\data}) \geq 1$ and $\alpha_V=\alpha_V ({\data})\in (0, 1)$
such that 
for all $\tau\in(0, 1]$ it holds
\begin{align*}
\mean{B_{\tau r}(x_0)}&|\teDu-(\teDu)_{B_{\tau r}(x_0)} | \dx\leq c_V \tau^{\alpha_V}
\mean{B_{r}(x_0)}|\teDu-(\teDu)_{B_{r}(x_0)} | \dx + c_V \sup_{0<\vr<r}g^{-1}\left(\frac{|\temu|(B_{\vr} (x_0))}{\vr^{n-1}}\right)\,.
\end{align*}
\end{lem}
\begin{proof} The proof follows \cite[Lemma~8.5]{KuMi2018}. We notice that it is enough to prove the claim for energy solutions to problems with smooth data and then pass to the limit according to the definition of SOLA, see Step 0 in the proof of Lemma~\ref{lem-8.4}. Suppose we are in this regime: $\teu$ is a weak solution and $\temu$ is smooth. In such situation we can make use of Lemmas~\ref{lem8.1}, \ref{lem8.2}, and \ref{lem8.3} to get that for every $j=0,1,\dots$ it holds
    \[E_{j+1}\leq\tfrac{1}{4}E_j+cg^{-1}\left(\frac{|\temu|(B^j)}{r_j^{n-1}}\right)\,,\]
for $c=c(\sigma,n,c_{\rm d},\theta)>0$. Iterating this inequality, by induction, and a manipulation, one can infer that
    \[E_{j+1}\leq\frac{\bar c}{4^k}\mean{B_r}|\teDu-(\teDu)_{B_r}|\dx+\bar c \sup_{0<\vr<r} g^{-1}\left(\frac{|\temu|(B_{\vr} (x_0))}{\vr^{n-1}}\right)\,,\]
which implies the claim by a standard interpolation argument presented in detail in the proof of \cite[Lemma~8.5]{KuMi2018} that does not require any change in our case. The constants are $\alpha_V:=[\log(1/4)]/\log\sigma$ and $c_V:=16(1+\bar c)\sigma^{-2n-1}$ if $\tau\in(0,\sigma r)$, whereas $\alpha_V:=-n$ and $c_V:=2$ if $\tau\in(\sigma r,1]$.
\end{proof}

\section{Proofs of main results}\label{sec:main-proofs}

Now, we are in the position to prove our main results Theorem \ref{theo:pointwise} and Theorem \ref{theo:VMOe}.

\begin{proof}[Proof of Theorem ~\ref{theo:VMOe}~(i)]
    The proof is a direct consequence of Lemma~\ref{lem-8.5}. For more details, see Step 1 in the proof of \cite[Theorem~1.2]{KuMi2018}.
\end{proof}

\begin{proof}[Proof of Theorem~\ref{theo:pointwise}] For   $\sigma$ as in~\eqref{sigma},  $H_1,H_2$ as in~\eqref{H1-and-H2} we define
    \begin{align}\label{Psi}
        \Psi(x_0,r) := \frac{64 H_1}{\sigma^n} \mean{B_r(x_0)} |\teDu - (\teDu)_{B_r(x_0)}| \dx + H_2 g^{-1}\left(\cI^{|\temu|}_1 (x_0 ,r)\right)\,.
    \end{align}
    Let us assume that $g^{-1}\left(\cI^{|\temu|}_1 (x_0 ,r)\right) < \infty$ implying, in particular, that $\Psi < \infty$. Moreover, by definition of $\Psi$ we infer that
    \begin{align}\label{8.42}
        g^{-1} \left( \sum_{j=0}^{\infty} \frac{|\temu|(B_{\vr} (x_0))}{\vr^{n-1}} \right) \le  g^{-1} \left(\frac{\sigma^{1-n}}{-\log\sigma} \sum_{j=0}^{\infty} \int_{r_j}^{r_{j-1}}\frac{|\temu|(B_{\vr} (x_0))}{\vr^{n-1}} \frac{{\rm d}\vr}{\vr}\right) \leq \sigma^{-n} \cI^{|\temu|}_1 (x_0 ,r)\leq \frac{\Psi}{H_2 \sigma^n}\,.
    \end{align}
    Thus, by Theorem~\ref{theo:VMOe}, we get
    \begin{align}\label{Psi-vanishes}
        \lim_{\vr \to 0} \left[ \frac{|\temu|(B_{\vr} (x_0))}{\vr^{n-1}} + \cI^{|\temu|}_1 (x_0 ,\vr) \right] = 0 \qquad\text{and}\qquad \lim_{\vr \to 0} \Psi(x_0,\vr) = 0\,.
    \end{align}
    Let us additionally observe that\begin{equation}\label{8.46}
    E_0\leq \frac{2}{\sigma^n}\mean{B_r(x_0)}|\teDu(x_0)-(\teDu)_{B_r(x_0)}|\dx\leq \frac{\Psi}{32H_1}\,.
\end{equation}
    
    \noindent{\bf Case 1. } Suppose $A_0>\Psi/16$. With the aim of showing~\eqref{Riesz-osc-est} we employ Lemma~\ref{lem-8.4} to in the proof that\begin{equation}
        \label{goal-in-case-1} |\teDu(x_0)-(\teDu)_{B_r(x_0)}|\leq \frac{2\Psi}{\sigma^n H_1}\,.
    \end{equation}
    The first step   is to show by induction that\begin{equation}\label{Aj-above-A0}
        A_j>A_0/2\qquad\forall j\geq 0\,.
    \end{equation}It follows the lines of Step~3 of the proof of \cite[Theorem~1.2]{KuMi2018}, so we show only the main arguments.
    
We  notice that\begin{equation}\label{8.47}
    C_0=A_0+H_1E_0\leq 2A_0 \,.
\end{equation}
By~\eqref{8.42} in this case we get \begin{equation*}
    g^{-1} \left( \sum_{j=0}^{\infty} \frac{|\temu|(B_{\vr} (x_0))}{\vr^{n-1}} \right) \leq \frac{\Psi}{H_2 \sigma^n}\leq \frac{16A_0}{H_2 \sigma^n}\,.
\end{equation*}
Therefore, by the triangle inequality, \eqref{8.46}, \eqref{8.47}, and the fact that $|A_1-A_0|\leq \sigma^{-n}E_0$, it follows that \eqref{Aj-above-A0} holds true for $j=1$. This estimate for $j>1$ can be proven by contradiction. Assume that there exists a finite index $J\geq 2$ such that
\[A_J\leq A_0/2\qquad \text{and}\qquad A_j>A_0/2\qquad\forall j\in\{0,\dots,J-1\}\,.\]
Then, precisely as in \cite[(8.52)-(8.53)]{KuMi2018}, with the use of Lemma~\ref{lem-8.4}, one can show that $C_j<8A_0$ for all $j\in\{0,\dots,J-1\}$. By collecting the remarks above and by the same manipulations oas in \cite[Step~3]{KuMi2018}, we get~\eqref{Aj-above-A0}. In turn, we are allowed to apply Lemma~\ref{lem-8.4} with ${k_0}=0$ and every $k$. Therefore, after letting $k\to\infty$, we obtain \[\sum_{j = 0}^{\infty} E_j \le 2 E_0 + \frac{H_2 \sigma^n}{64 H_1} \frac{1}{g'(8A_0)} \sum_{j = 0}^\infty \frac{|\temu|({B^j})}{r_j^{n-1}}\,.\]
Since $g'$ is increasing, we know that in this case 
\begin{align}
\label{eq:C-large2}
    \frac{1}{g'(8A_0)} g \left( \frac{\Psi}{H_2 \sigma^n} \right) \leq \frac{g(\Psi)}{g'(\Psi)} 2^{q-2} (H_2\sigma^n)^{-(p-1)} \leq \frac{2^{q-2}(H_2\sigma^n)^{-(p-1)}}{p-1} \Psi
\end{align}
and we have  \[\sum_{j = 0}^{\infty} E_j \le 2 E_0 + \frac{H_2 \sigma^n}{64 H_1} \frac{1}{g'(8A_0)} g\left(\frac{\Psi}{H_2 \sigma^n}\right)\leq \frac{\Psi}{32 H_1}+\frac{C(p,q)(H_2\sigma^n)^{2-p}}{ H_1}\Psi\leq \frac{\Psi}{H_1}\,\]
for some $C(p,q) > 0$, where the last estimate follows by choosing the constant $C(\data) > 0$ in the definition of $H_2$ large enough depending on $C(p,q)$.
In fact, knowing that for ${k_0}<k$ it holds $|(\teDu)_{B^k}-(\teDu)_{B^{k_0}}|\leq \sigma^{-n}\sum_{j={k_0}}^\infty E_j\leq\sigma^{-n}\Psi/H_1$, the above estimate implies that $\{(\teDu)_{B^j}\}_j$ is a Cauchy sequence. Consequently, $x_0$ is a Lebesgue's point of $\teDu$, that is $\lim_{\vr\to 0}(\teDu)_{B_\vr(x_0)}=\teDu(x_0)$, and
\[|\teDu(x_0)-(\teDu)_{B_{\sigma r}(x_0)}|\leq \sigma^{-n}\sum_{j=0}^\infty E_j\,.\]
Therefore\begin{align*}
    |\teDu(x_0)-(\teDu)_{B_{r}(x_0)}|&\leq |\teDu(x_0)-(\teDu)_{B_{\sigma r}(x_0)}|+|(\teDu)_{B_{\sigma r}(x_0)}-(\teDu)_{B_{r}(x_0)}|\\
    &\leq \sigma^{-n}\sum_{j=0}^\infty E_j+\sigma^{-n}\mean{B_r(x_0)}|\teDu-(\teDu)_{B_{r}(x_0)}|\dx\leq \frac{2\Psi}{\sigma^n H_1}\,,
\end{align*}
which is~\eqref{goal-in-case-1}.
    
    \noindent{\bf Case 2. } Suppose $A_0\leq\Psi/16$ for $\Psi$ defined in~\eqref{Psi}. Without loss of generality we can assume that $\Psi>0$. Indeed, otherwise $\teDu$ is constant in $B_r(x_0)$, $A_0=0$, and there is nothing to prove. Note that
    \begin{equation*}
        C_0=A_0+H_1E_0\leq \Psi/16+\Psi/32<\Psi/8\,.
    \end{equation*}
    Arguing by induction and contradiction as in \cite[Step~4]{KuMi2018} with the use of Lemma~\ref{lem-8.4} instead of \cite[Lemma~8.4]{KuMi2018}, we learn that\begin{equation*}
        C_j<\Psi/2\,.
    \end{equation*}
     Let us take $\vr\in(0,\sigma r]$ and $k\geq 0$ being the largest integer for which $r_k=\sigma^{k+1}\geq\vr$. Then $(r_k/\vr)^{-n}\leq (r_k/r_{k+1})^{-n}=\sigma^{-n}$ and
     \begin{align*}
         |(\teDu)_{B_\vr(x_0)}|&\leq A_k+|(\teDu)_{B_\vr(x_0)}-(\teDu)_{B^k}|\leq A_k+\sigma^{-n}\mean{B^k}|\teDu-(\teDu)_{B^k}|\dx\leq C_k\leq \Psi/2\,.
     \end{align*}
     If $\vr\in(\sigma r,1]$, then we make use of the fact that $A_0\leq\Psi$ to deduce  \begin{align*}
         |(\teDu)_{B_\vr(x_0)}|&\leq A_0+|(\teDu)_{B_\vr(x_0)}-(\teDu)_{B^0}|\leq A_0+\sigma^{-n}\mean{B_r(x_0)}|\teDu-(\teDu)_{B_r(x_0)}|\dx\leq \Psi/2\,.
     \end{align*}
     Altogether, we get\begin{equation}
         \sup_{0<\vr<r} |(\teDu)_{B_\vr(x_0)}| \leq \Psi(x_0,r)/2\,.\label{shrink}
     \end{equation}
     
     To justify that $x_0$ is a Lebesgue's point we assume by contradiction that the limit $\lim_{\vr\to 0}(\teDu)_{B_\vr(x_0)}$ does not exist. This means that $|(\teDu)_{B_{\vr}(x_0)}|\leq \Psi(x_0,r)/16$ for every $\vr<\sigma r$. Indeed, otherwise one can show that the limit exists via the arguments from Case~1. Consequently, since we know by~\eqref{Psi-vanishes} that $\Psi(x_0,\vr)\to 0$ as $\vr\to 0$,~\eqref{shrink} implies that $(\teDu)_{B_\vr(x_0)}\to 0$ as $\vr\to 0$ and we get the desired contradiction. Hence, $x_0$ is a Lebesgue's point of $\teDu$.
     
     To prove~\eqref{Riesz-osc-est}, we collect the above information and get
     \[|(\teDu)_{B_{\vr} (x_0 )}-(\teDu)_{B_r (x_0 )} |\leq \sup_{0<\vr<r} |(\teDu)_{B_{\vr}(x_0 )} | + |(\teDu)_{B_r (x_0 )} | \leq \Psi(x_0 , r)\,.\]
     We can pass to the limit with $\vr\to 0$. Since $x_0$ is a Lebesgue's point, \eqref{Riesz-osc-est} follows.

    \noindent{\bf Conclusion. } In both cases above~\eqref{Riesz-osc-est} is proven. Then~\eqref{eq:u-est} follows from the triangle inequality.
\end{proof}

\begin{proof}[Proof of Theorem \ref{theo:VMOe}~(ii)]
        The proof is a direct consequence of \eqref{Riesz-osc-est} in Theorem~\ref{theo:pointwise}. It relies on restricting the radius of a ball centred in $x_0$ twice to make $|\teDu-(\teDu)_{r_\delta}|$ arbitrarily small in the limit as $r_\delta\to 0$. For more details, see \cite{KuMi2018}.
\end{proof}

\begin{appendix}
    
\label{sec-appendix}

\section{Properties of growth functions}

We collect some basic properties of growth functions that we use throughout the course of this article.

\begin{lem}
    \label{lem-g'}
    Suppose $f:[0,\infty)\to[0,\infty)$ is non-decreasing and $\texi,\teeta \in \rnm$, then
    \begin{equation*}
            \frac{1}{3}f\left(\frac{|\texi|+|\teeta|}{12}\right)\leq \int_0^1 f(|\texi+t\teeta|)\dt\leq f(|\texi|+|\teeta|)\,.
    \end{equation*}
\end{lem}
\begin{proof}
    The upper bound follows from the triangle inequality. For the lower bound, we observe the following pointwise estimates:
    \begin{align*}
        \begin{cases}
            |\texi + t \teeta| \ge |\texi| - t |\teeta| \ge |\texi| - \frac{1}{3}|\teeta| \ge \frac{1}{12} (|\texi| + |\teeta|)\,, \quad\quad\quad \ \ |\texi| \ge \frac{1}{2}|\teeta|\,,~~ t \le \frac{1}{3}\,,\\
            |\texi + t \teeta| \ge t|\teeta| - |\texi| \ge \frac{1}{6}|\teeta| \ge \frac{1}{12}(|\texi| + |\teeta|)\,, \qquad \qquad \qquad |\texi| \le \frac{1}{2}|\teeta|\,,~~  t \ge \frac{2}{3}\,.
        \end{cases}
    \end{align*}
    Therefore, we see that
    \begin{align*}
    \int_0^1 f(|\texi+t\teeta|)\dt \geq
        \begin{cases}
            \int_{0}^{1/3}  f\left(\frac{|\texi|+|\teeta|}{12}\right)\dt\,, \quad |\texi|\geq \frac{1}{2}|\teeta|\\
            \int_{2/3}^1  f\left(\frac{|\texi|+|\teeta|}{12}\right)\dt\,, \quad |\texi|\leq \frac{1}{2}|\teeta|
        \end{cases}
        \geq \frac{1}{3} f\left(\frac{|\texi|+|\teeta|}{12}\right)\,.
    \end{align*}
\end{proof}

\begin{lem}
    \label{lem-G}
    For $g$ be such that \eqref{a:p-1} holds true. Then, $t^2 {g(1+t)}/{(1+t)}\approx t^2+G(t)$ for all ${t\geq 0}$.
\end{lem}

\begin{lem}
\label{lem-g''}
Let $g$ be such that \eqref{a:p-1} and \eqref{a:p-2} hold true. Then, for every $\texi,\teeta \in \rnm$ it holds
    \begin{align*}
    \int_0^1 \int_{0}^{1} g''(|\texi + \tau s (\teeta - \texi)|) \dtau \ds \leq c g''(|\texi|+|\teeta|)\,,
\end{align*}
where $c > 0$ depends only on $p,q$.
\end{lem}

\begin{proof}
    We start by observing that there exists a constant $c > 0$ depending only on $p,q$ such that
 \begin{align}
 \label{eq:aux-lemma-claim}
        \frac{|g'(a) - g'(b)|}{g''(a+b)|a - b|} \le c
    \end{align}
    for any $a,b > 0$. Without loss of generality,  let us assume that $a < b$ and write $b = \kappa a$ for some $\kappa > 1$. Then, we can write using \eqref{a:p-2} and the monotonicity of $g'$:
    \begin{align*}
        \frac{|g'(a) - g'(b)|}{g''(a+b)|a - b|} = \frac{|g'(a) - g'(\kappa a)|}{g''((1  + \kappa) a)(\kappa - 1) a} \le  (p-2)^{-1} \frac{1+\kappa}{\kappa - 1} \frac{|g'(a) - g'(\kappa a)|}{g'(\kappa a)}\,.
    \end{align*}
    Let us first assume that $2 \le \kappa$. In that case,
    \begin{align*}
        \frac{|g'(a) - g'(b)|}{g''(a+b)|a - b|} \le (p-2)^{-1} \frac{1+\kappa}{\kappa - 1} \frac{g'(\kappa a)}{g'((1+\kappa)a)} \le \frac{3}{2}(p-2)^{-1}\,.
    \end{align*}
    In case $1 \le \kappa \le 2$, we obtain
    \begin{align*}
        \frac{|g'(a) - g'(b)|}{g''(a+b)|a - b|} \le (p-2)^{-1} \frac{\kappa^{q-2} - 1}{\kappa - 1} \frac{(1+\kappa) g'(a)}{g'((1+\kappa)a)} \le 3(p-2)^{-1}c(q)\,,
    \end{align*}
    where $\sup_{\kappa \le [1,2]} \frac{\kappa^{q-2} - 1}{\kappa - 1} =: c(q) > 0$ depends only on $q$. This proves the claim \eqref{eq:aux-lemma-claim}.

    In order to explain how \eqref{eq:aux-lemma-claim} implies the desired result, we compute
    \begin{align*}
        \int_0^1 \int_{0}^{1} g''(|\texi + \tau s (\teeta - \texi)|) \dtau \ds
        &= \int_0^1 \frac{1}{s|\teeta - \texi|} \left[ g'(|\texi + s(\teeta-\texi)|) - g'(|\texi|) \right] \ds\\
        &\le \int_0^1 \frac{1}{||\texi +s(\teeta - \texi)| - |\texi||} \left[ g'(|\texi + s(\teeta-\texi)|) - g'(|\texi|) \right] \ds\,.
    \end{align*}
    Note that upon application of \eqref{eq:aux-lemma-claim} with $a = |\texi + s(\teeta-\texi)|$ and $b = |\texi|$, we obtain
    \begin{align*}
        \int_0^1 \int_{0}^{1} g''(|\texi + \tau s (\teeta - \texi)|) \dtau \ds &\le c \int_0^1 g''(|\texi + s(\teeta-\texi)| + |\texi|) \ds \\
        &\le c \frac{|g'(|\teeta| + |\texi|) - g'(2|\texi|)|}{|\teeta - \texi|}\le c \frac{|g'(|\teeta| + |\texi|) - g'(2|\texi|)|}{|(|\teeta| + |\texi|) - 2|\texi||}\,.
    \end{align*}
    By application of \eqref{eq:aux-lemma-claim} with $a = |\teeta| + |\texi|$ and $b = 2|\texi|$, we obtain
    \begin{align*}
        \int_0^1 \int_{0}^{1} g''(|\texi + \tau s (\teeta - \texi)|) \dtau \ds &\le c g''(|\teeta| + 3|\texi|)\le c \frac{g'(|\teeta|+3|\texi|)}{|\teeta| + 3|\texi|}\le c \frac{g'(3(|\teeta|+|\texi|))}{|\teeta|+|\texi|}\le c g''(|\teeta|+|\texi|)\,,
    \end{align*}
    where we also used monotonicity of $g'$ and \eqref{a:p-2}. This proves the desired result.
    \end{proof}

\section{Function spaces}\label{ssec:fnsp}
Before presenting the definitions of Lorentz, Morrey, and Lorentz--Morrey spaces,  we recall the definitions of decreasing rearrangement and maximal rearrangement of a function (see ~\cite{BeSa88}). The decreasing rearrangement $f^* : [0, \infty) \to [0, \infty]$ of a measurable function $f:\rn\to\R$ is given by
\begin{equation*}
f^\ast(t) := \sup\Big \{ s \geq 0 \colon \big|\{x\in \R^n:|f(x)|>s\}\big| > t  \Big\}\,,
\end{equation*}
where we assume that $\sup\emptyset=0$. Equivalently it can also be defined as the function, which is right-continuous, non-increasing, and equimeasurable with $f$, i.e., $|\{x : |f(x)| > t\}| = |\{x: f^{*}(x) > t\}|$ for all $t > 0$.
 The maximal rearrangement is defined by
\begin{equation*}
f^{**}(t): =\frac 1t \int_0^t f^\ast(s) \, ds\quad\text{and}\quad 
f^{\ast \ast}(0)=
f^{\ast}(0)\,.
\end{equation*}

\noindent{}Let $0 < r,s \leq \infty$. The Lorentz space $\Lambda^{r,s}(\Omega)$ is the space of measurable functions $f : \Omega \to \R$ such that
\begin{equation*}
\|f\|_{{\Lambda^{r,s}}(\Omega)}:= 
\big\|(\cdot)^{\frac 1r- \frac 1s}f^{**}(\cdot)\big\|_{L^s(0, |\Omega|)}
< \infty\,.
\end{equation*}
 For $s=\infty$ this space is usually called the Marcinkiewicz space. 

\noindent We say that  a measurable function $f : \Omega\to\rn$ belongs to the Morrey space $L^{r}_{\vt}(\Omega)$ for $r\geq 1$ and $\vt\in[0,n]$, 
if and only if \[\|f\|_{L^{r}_{\vt}(\Omega)}:= \sup\limits_{\substack{
B(x_0,R) \subset\rn\\ x_0\in\Omega}} R^{\frac{\vt-n}{r}} \|f\|_{L^r(B_R\cap \Omega)} <\infty\,.\]
 We say that $f : \Omega \to \rn$ belongs to the Lorentz--Morrey space $\Lambda^{r, s}_{\vt}(\Omega)$ for $r,s\in[1,\infty]$ and $\vt\in[0,n]$,
if and only if
\[\|f\|_{\Lambda^{r,s}_{\vt}(\Omega)}:=\sup\limits_{\substack{
B(x_0,R) \subset\rn\\ x_0\in\Omega}}  R^\frac{\vt-n}{r}\|f\|_{\Lambda^{r,s}(B_R\cap \Omega)}<\infty\,.\]

\noindent We define the space $L\log  L(\Omega)$ as a subset of integrable functions $f:\Omega\to\R$ such that 
\[\begin{split}\|f\|_{L \log  L(\Omega)}&=\inf\left\{\lambda>0:\quad\int_\Omega\Big|\tfrac{f}{\lambda}\Big|\log\left(e +\Big|\tfrac{f}{\lambda}\Big| \right)dx\leq 1\right\} <\infty\,.\end{split}\]
\end{appendix}
\noindent{\bf Acknowledgements} I. Chlebicka is supported by NCN grant no. 2019/34/E/ST1/00120. M. Kim is supported by the research fund of Hanyang University (HY-202300000001143) and the National Research Foundation of Korea (NRF) grant funded by the Korea government (MSIT) (RS-2023-00252297). M. Weidner is supported by the European Research Council (ERC) under the Grant Agreement No 801867 and the AEI project PID2021-125021NA-I00 (Spain). 
The project started with discussions of all authors during Thematic Research Programme Anisotropic and Inhomogeneous Phenomena funded by The Initiative of Excellence at University of Warsaw in 2022.

\def\ocirc#1{\ifmmode\setbox0=\hbox{$#1$}\dimen0=\ht0 \advance\dimen0
  by1pt\rlap{\hbox to\wd0{\hss\raise\dimen0
  \hbox{\hskip.2em$\scriptscriptstyle\circ$}\hss}}#1\else {\accent"17 #1}\fi}
  \def\cprime{$'$} \def\ocirc#1{\ifmmode\setbox0=\hbox{$#1$}\dimen0=\ht0
  \advance\dimen0 by1pt\rlap{\hbox to\wd0{\hss\raise\dimen0
  \hbox{\hskip.2em$\scriptscriptstyle\circ$}\hss}}#1\else {\accent"17 #1}\fi}
  \def\ocirc#1{\ifmmode\setbox0=\hbox{$#1$}\dimen0=\ht0 \advance\dimen0
  by1pt\rlap{\hbox to\wd0{\hss\raise\dimen0
  \hbox{\hskip.2em$\scriptscriptstyle\circ$}\hss}}#1\else {\accent"17 #1}\fi}
  \def\cprime{$'$}

\end{document}